\documentclass[10pt]{article}
\usepackage{color}
\usepackage{amssymb}
\usepackage{amsthm,array,amssymb,amscd,amsfonts,latexsym, url}
\usepackage{amsmath}
\usepackage[all]{xy}
\newtheorem{theo}{Theorem}[section]
\newtheorem{prop}[theo]{Proposition}
\newtheorem{claim}[theo]{Claim}
\newtheorem{lemm}[theo]{Lemma}
\newtheorem{coro}[theo]{Corollary}
\newtheorem{question}[theo]{Question}
\newtheorem{rema}[theo]{Remark}
\newtheorem{Defi}[theo]{Definition}

\title{On the coniveau of rationally connected threefolds}
\author{Claire Voisin}

\date{}

\newfont{\gothic}{eufb10}
\begin{document}
\maketitle
\setcounter{section}{-1}

\begin{abstract} We prove that the {\it integral} cohomology modulo torsion of a rationally connected threefold comes from
the integral cohomology of a smooth curve via the cylinder homomorphism associated to a family of $1$-cycles. Equivalently, it is of {\it strong} coniveau $ 1$. More generally, for a rationally connected manifold $X$ of dimension $n$, we show that  the
strong coniveau $\widetilde{N}^{n-2}H^{2n-3}(X,\mathbb{Z})$ and coniveau $N^{n-2}H^{2n-3}(X,\mathbb{Z})$ coincide for  cohomology modulo torsion.
 \end{abstract}
\section{Introduction}
We work over $\mathbb{C}$ and cohomology is Betti cohomology.  Given an abelian group $A$, recall that a cohomology class $\alpha\in H^k(X,A)$  has coniveau $\geq c$ if $\alpha_{\mid U}=0$ for some Zariski open set
$U=X\setminus Y$, with ${\rm codim}\,Y\geq c$. Equivalently, $\alpha$ comes from the relative cohomology
$H^k(X,U,A)$.
If $X$ is smooth projective of dimension $n$ and $A=\mathbb{Z}$, using Poincar\'e duality,
$\alpha\in H_{2n-k}(X,\mathbb{Z})$ comes from a homology class on $Y$
\begin{eqnarray}
\label{eqalphabeta} \alpha=j_* \beta\,\,{\rm in}\,\, H_{2n-k}(X,\mathbb{Z}),
\end{eqnarray}
for some $\beta \in H_{2n-k}(Y,\mathbb{Z})$.
 In the situation above, the closed algebraic subset  $Y$ cannot in general taken to be smooth.
Take for  example a smooth hypersurface  $X\subset \mathbb{P}^{n+1}$ with $n$ odd, $n\geq 3$. Then
$\rho(X)=1$ and   by the Lefschetz theorem on hyperplane sections,
for any smooth hypersurface $Y\subset  X$, $H^{n-2}(Y,\mathbb{Z})=0$, so no degree $n$  cohomology class on $X$ is supported
on a smooth hypersurface.
One can wonder however if, in the situation above,  after taking a desingularization $\tau:\widetilde{Y}\rightarrow Y$ of $Y$, with
composite map $\tilde{j}=j\circ \tau:\widetilde{Y}\rightarrow X$, one can rewrite (\ref{eqalphabeta}) in the form
\begin{eqnarray}
\label{eqalphabeta2} \alpha=\tilde{j}_* \tilde{\beta}\,\,{\rm in}\,\, H_{2n-k}(X,A).
\end{eqnarray}
 In the situation described above, when $X$ is smooth projective,  Deligne \cite{deligne} shows that, with $\mathbb{Q}$-coefficients,  $${\rm Im}\,(\tilde{j}_*: H_{2n-k}(\widetilde{Y},\mathbb{Q})\rightarrow H_{2n-k}(X,\mathbb{Q}))=
{\rm Im}\,({j}_*: H_{2n-k}({Y},\mathbb{Q})\rightarrow H_{2n-k}(X,\mathbb{Q})),$$
so that the answer is yes with $\mathbb{Q}$-coefficients.
With $\mathbb{Z}$-coefficients, this  is wrong, as shows the following simple   example: Let $j': \widetilde{C}\hookrightarrow   A$ be  a smooth genus $2$ curve in an abelian surface.
Let $\mu_2:A\rightarrow A$ be the multiplication by $2$ and let
$C=\mu_2(\widetilde{C})\subset A$, with inclusion map $j: C\rightarrow A$. As $j(C)$ is an ample curve, the Lefschetz theorem on hyperplane sections
says that $j_*:H_1(C,\mathbb{Z})\rightarrow H_1(A,\mathbb{Z})$ is surjective. However,
$C$ admits $\tilde{j}:=\mu_2\circ j': \widetilde{C}\rightarrow A$ as normalization and the
map $\tilde{j}_*:H_1(\widetilde{C},\mathbb{Z})\rightarrow H_1(A,\mathbb{Z})$ is not surjective as
 $\tilde{j}_*=2j'_*$ so ${\rm Im}\,\tilde{j}_*$ is contained in $2H_1(A,\mathbb{Z})$.

 In this example, the degree $1$ homology of $A$ (or degree $3$ cohomology) is however  supported on smooth curves.
To follow the terminology introduced by Benoist and Ottem in \cite{BeOt}, let us say that a cohomology class $\alpha\in H^k(X,\mathbb{Z})$
on a smooth projective complex  manifold $X$ is of strong coniveau $\geq c$ if there exists
a {\it smooth} projective manifold of dimension $n-c$, and a morphism
$f:Y\rightarrow X$ such that $\alpha =j_*\beta$ for some cohomology class $\beta \in H^{k-2c}(Y,\mathbb{Z})$. ($Y$ being smooth, we can apply Poincar\'{e} duality and use the Gysin morphism in cohomology.)
Benoist and Ottem prove the following result.
\begin{theo} (Benoist-Ottem, \cite{BeOt}) If $c\geq 1$ and $k \geq 2c+1$, there exist  complex projective  manifolds $X$  and  integral cohomology classes of degree $k$ on $X$ which are  of coniveau $\geq c$  but not of strong coniveau $\geq c$.\end{theo}

Their construction however imposes restrictions on the dimension of $X$ and for example, the case where
$k=3,\,c=1$, ${\rm dim}\,X=3$ remains open. For $c=1$, the examples constructed in \cite{BeOt} are varieties of general type.

We study in this paper the case of rationally connected $3$-folds (and more generally degree $3$ homology of rationally connected manifolds). As we will recall in Section \ref{seccomments}, the integral cohomology of
degree $>0$ of a smooth complex projective  rationally connected manifold is of coniveau $\geq1$. However, except in specific cases, there are no general available results for the strong coniveau.
Our main result  is the following.

\begin{theo}\label{theomain} Let $X$ be a smooth projective rationally connected threefold over $\mathbb{C}$. Then
the cohomology $H^3(X,\mathbb{Z})$ modulo torsion has strong coniveau $1$.
\end{theo}
 It turns out
that an equivalent formulation is the following
\begin{coro} \label{corintro} If $X$ is  a rationally connected threefold, there exist  a smooth curve $C$  and   a family of $1$-cycles
$\mathcal{Z}\in {\rm CH}^2(C\times X)$  such that
 the cylinder homomorphism $[\mathcal{Z}]_*:H^1(C,\mathbb{Z})\rightarrow H^3(X,\mathbb{Z})_{\rm tf}$  is surjective.
\end{coro}
Here and in the sequel, we denote $\Gamma_{\rm tf}:=\Gamma/{\rm Torsion}$ for any abelian group  $\Gamma$.
\begin{proof} (See more generally Proposition \ref{procylegalstrong}.) Theorem \ref{theomain} says that there exist a smooth projective surface $\Sigma$ and a morphism
$f:\Sigma\rightarrow X$ such that $f_*:H^1(\Sigma,\mathbb{Z})\rightarrow H^3(X,\mathbb{Z})_{\rm tf}$ is surjective.
The existence of a Poincar\'e divisor  $\mathcal{D}\in {\rm CH}^1({\rm Pic}^0(\Sigma)\times \Sigma)$, satisfying the property
that $[\mathcal{D}]_*: H_1({\rm Pic}^0(\Sigma),\mathbb{Z})\rightarrow H^1(\Sigma,\mathbb{Z})$ is an isomorphism, provides
a codimension $2$-cycle  $$\mathcal{Z}=(Id,f)_*(\mathcal{D})\in{\rm CH}^2( {\rm Pic}^0(\Sigma)\times X)$$ such that
$$[\mathcal{Z}]_*: H_1({\rm Pic}^0(\Sigma),\mathbb{Z})\rightarrow H^3(X,\mathbb{Z})_{\rm tf}$$
is surjective. We finally choose any smooth  curve $C$ complete intersection of ample hypersurfaces  in  ${\rm Pic}^0(\Sigma)$ and restrict $\mathcal{Z}$ to $C$. The corollary then follows by the Lefschetz hyperplane section  theorem applied to $C\subset {\rm Pic}^0(\Sigma)$.
\end{proof}
This theorem will be proved in Section \ref{secfinal}.  We will prove in fact more generally (see Theorem \ref{theomainvrai})
\begin{theo} \label{theointronew} For any rationally connected smooth projective variety
of dimension $n$, one has the equality
$$N^{n-2} H^{2n-3}(X,\mathbb{Z})_{\rm tf}=\widetilde{N}^{n-2}H^{2n-3}(X,\mathbb{Z})_{\rm tf}.$$
Furthermore, the equality
 $$H^{2n-3}(X,\mathbb{Z})_{\rm tf}=\widetilde{N}^{n-2}H^{2n-3}(X,\mathbb{Z})_{\rm tf}$$
 holds
 assuming that the Abel-Jacobi map
 $\Phi_X:{\rm CH}_1(X)_{\rm alg}\rightarrow J^{2n-3}(X)$
 is injective on torsion.
 \end{theo}
 This last assumption, which is automatically satisfied when $n=3$,  is related to the  following question mentioned in  \cite[1.3.3]{voisingargnano}).
\begin{question} Let $X$ be a rationally connected manifold. Is the Abel-Jacobi map
$\Phi_X:{\rm CH}_1(X)_{\rm alg}\rightarrow J^{2n-3}(X)$
 injective on torsion cycles?
\end{question}
Note that, as explained in {\it loc. cit.}, the group  $${\rm Tors}({\rm Ker}\,(\Phi_X: {\rm CH}_1(X)_{\rm alg}\rightarrow J^{2n-3}(X)))$$ is a stable birational invariant of   projective complex manifolds $X$, which is trivial when $X$ admits a Chow decomposition of the diagonal. Results of
Suzuki \cite{suzuki} give a complete understanding of this birational invariant in terms of coniveau (see also Section \ref{secaj}).

 In Section \ref{secmaterial}, we discuss
various notions of coniveau in relation to rationality or stable rationality questions,  which we will need to split the statement of the main theorems  into two different statements.  In particular we introduce the     ``cylinder homomorphism filtration" $N_{c,{\rm cyl}}$, which has a strong version $\widetilde{N}_{c,{\rm cyl}}$. The cylinder homomorphism filtration $N_{c,{\rm cyl}}H_{k+2c}(X,\mathbb{Z})$ on the homology (or cohomology) of  a smooth projective manifold $X$ uses proper  flat families
$\mathcal{Z}\rightarrow Z$
of  subschemes of $X$ of  dimension $c$, and the associated cylinder map
$H_k(Z,\mathbb{Z})\rightarrow H_{k+2c}(X,\mathbb{Z})$, which by flatness can be defined without any smoothness assumption on $Z$.
The strong version $\widetilde{N}_{c,{\rm cyl}}H_{k+2c}(X,\mathbb{Z})$ is similar but imposes the smoothness assumption to $Z$ (so flatness is not needed anymore).
 When $c=1$,
it is better to use  the stable-cylinder filtration $N_{1,{\rm cyl,st}}H_{k+2}(X,\mathbb{Z})$ (where $X$ is  mooth projective of dimension $n$), which  is generated by
the cylinder homomorphisms
$$H_k(Z,\mathbb{Z})\rightarrow H_{k+2}(X,\mathbb{Z})$$
for all families of semi-stable maps from curves to  $X$, without smoothness assumption on $Z$ (but we will assume that ${\rm dim}\,Z\leq k$).
These various filtrations and their inclusions are discussed   in Section \ref{secmaterial}.
We prove in Section \ref{seccyltheo}  Theorem \ref{theocylinder}, which is the first step towards the proof of Theorem
\ref{theomain}, and in dimension $3$ says the following.
\begin{theo}\label{theocylintro} (Cf. Corollary \ref{corocylinder}) Restricting to the  torsion free part of cohomology, one has
the equality
$$N_{1,{\rm cyl,st}}H^3(X,\mathbb{Z})_{\rm tf}=N^1H^3(X,\mathbb{Z})_{\rm tf}$$
for any smooth projective complex threefold $X$.
\end{theo}
The second step of the proof of Theorem
\ref{theomain} is the following  result, now valid for rationally connected manifolds of any dimension and also for the torsion part of homology.
\begin{theo} (Cf. Theorem \ref{theocylRC}) Let $X$ be rationally connected smooth projective over $\mathbb{C}$.
Then
\begin{eqnarray} N_{1,{\rm cyl,st}}H^{2n-3}(X,\mathbb{Z})=\widetilde{N}_{1,{\rm cyl}}H^{2n-3}(X,\mathbb{Z}).
\end{eqnarray}
Equivalently, $N_{1,{\rm cyl,st}}H^{2n-3}(X,\mathbb{Z})= \widetilde{N}^{n-2}H^{2n-3}(X,\mathbb{Z})$.
\end{theo}
{\bf Thanks.} {\it  I thank Fumiaki Suzuki for reminding me his results for $1$-cycles in \cite{suzuki}, which improved Theorem \ref{theointronew}
and  removed an assumption in  Theorem \ref{theomainvrai}.}
\section{Various notions of niveau and  coniveau \label{secmaterial}}
We are going to discuss in this section another filtration on cohomology,  namely the (strong)
cylinder homomorphism
filtration (which is better understood in homology, so that we will speak of niveau) with emphasis on the niveau  $1$. It is particularly interesting in the case of   niveau $1$  because
we will be able in this case  to extract from this  definition  further stable birational invariants, which is not the case for higher niveau.
We will work with Betti cohomology with integral coefficients and our varieties $X$ will be smooth projective of dimension $n$ over $\mathbb{C}$.
We already mentioned in the introduction the coniveau filtration
$N^cH^k(X,\mathbb{Z})$ and the strong coniveau filtration $\widetilde{N}^cH^{k}(X,\mathbb{Z})$. By definition,
$\widetilde{N}^cH^{k}(X,\mathbb{Z})$ is generated by  the images $\Gamma_*H^{k-2c}(Y,\mathbb{Z})$, for
all smooth projective varieties $Y$ of dimension $n-c$ and all
 morphisms $\Gamma:Y\rightarrow X$, or more generally  codimension $n$ correspondences $\Gamma\in{\rm CH}^n(Y\times X)$.

We  now introduce a different filtration,
\begin{eqnarray}\label{eqdefcyl}
\widetilde{N}_{c,{\rm cyl}}H^{k}(X,\mathbb{Z})\subset H^{k}(X,\mathbb{Z}),
\end{eqnarray}
that we will call the strong cylinder homomorphism filtration (see \cite{shimada}).
\begin{Defi}
\label{deficyl} We denote by  $\widetilde{N}_{c,{\rm cyl}}H^{k}(X,\mathbb{Z})\subset H^{k}(X,\mathbb{Z})$  the subgroup
of $H^{k}(X,\mathbb{Z})$ generated by the images of the cylinder homomorphisms
\begin{eqnarray}\label{eqcylhomo2oct} \Gamma_*:H_{2n-k-2c}(Z,\mathbb{Z})\rightarrow H_{2n-k}(X,\mathbb{Z})= H^{k}(X,\mathbb{Z}),
\end{eqnarray} for all smooth projective varieties $Z$  and
correspondences $\Gamma\in {\rm CH}^{n-c}(Z\times X)$.
\end{Defi}
We will occasionally  use the notation $\widetilde{N}_{c,{\rm cyl}}H_{k}(X,\mathbb{Z})\subset H_{k}(X,\mathbb{Z})$ for the corresponding filtration on homology, which is in fact more natural.
We can think to $\Gamma$ as a family of cycles of dimension $c$ in $X$ parameterized by $Z$.
\begin{lemm} \label{lequisert} We have $\widetilde{N}_{c,{\rm cyl}}H^{k}(X,\mathbb{Z})\subset \widetilde{N}^{k+c-n}H^{k}(X,\mathbb{Z})$. In particular, for $k=n$, we have $\widetilde{N}_{1,{\rm cyl}}H^n(X,\mathbb{Z})\subset \widetilde{N}^{1}H^n(X,\mathbb{Z})$.
\end{lemm}
\begin{proof} In the definition \ref{deficyl}, we observe that, as $Z$ is smooth, by the Lefschetz
theorem on hyperplane section, its homology of degree $2n-k-2c$ is
supported on smooth subvarieties $Z'$ of $Z$ of dimension $\leq 2n-k-2c$. It follows that we can restrict in (\ref{eqcylhomo2oct}) to the case where
${\rm dim}\,Z\leq 2n-k-2c$.
The inclusion $\widetilde{N}_{c,{\rm cyl}}H^{k}(X,\mathbb{Z})\subset \widetilde{N}^{k+c-n}H^{k}(X,\mathbb{Z})$ then follows from the fact that, by desingularization,
cycles $\Gamma\in {\rm CH}^{n-c}(Z\times X)$  can be chosen to be represented by combinations with integral coefficients of smooth projective varieties $\Gamma_i$ mapping to $ Z \times X$, so that
$${\rm Im}\,\Gamma_*\subset \sum_i{\rm Im}\,\Gamma_{i*}.$$
As ${\rm dim}\,Z\leq 2n-k-2c$ and ${\rm codim}\,(\Gamma_i/ Z\times X)=n-c$, we  have
 ${\rm dim}\,\Gamma_i\leq 2n-k-c$, so that, by definition,
$${\rm Im}\,\Gamma_{i*}\subset \widetilde{N}^{k+c-n}H^k(X,\mathbb{Z}).$$
\end{proof}

With $\mathbb{Q}$-coefficients, the  definition (\ref{eqdefcyl}) appears  in
\cite{vial}. For $k=n$ and $\mathbb{Q}$-coefficients,  the Lefschetz standard conjecture for smooth projective varieties $Y$ of dimension $n-c$ and for degree $n-2c$ predicts that
\begin{eqnarray} \label{eqagalef}  \widetilde{N}_{c,{\rm cyl}}H^{n}(X,\mathbb{Q})= \widetilde{N}^cH^{n}(X,\mathbb{Q}).
\end{eqnarray}
Indeed, the hard Lefschetz theorem  gives for any smooth projective variety  $Y$ of dimension $n-c$ the
hard Lefschetz isomorphism
$$L^{c}:H^{n-2c}(Y,\mathbb{Q})\cong H^n(Y,\mathbb{Q})$$
where the Lefschetz operator $L$ is the cup-product operator with the class $c_1(H)$ for some very ample divisor $H$ on $Y$,
 and the  Lefschetz standard conjecture predicts the existence of a codimension-$n-2c$ cycle
 $\mathcal{Z}_{\rm Lef}\in{\rm CH}^{n-2c}(Y\times Y)_\mathbb{Q}$ such that
 $$[ \mathcal{Z}_{\rm Lef}]_*\circ L^c=Id:  H^{n-2c}(Y,\mathbb{Q})\rightarrow H^{n-2c}(Y,\mathbb{Q}).$$
 Restricting $\mathcal{Z}_{\rm Lef}$ to $Z\times Y$, where $Z\subset Y$ is a smooth complete intersection of
 $c$  ample hypersurfaces
 in $|H|$, we get a cycle
 $$\mathcal{Z}'_{\rm Lef}\in{\rm CH}^{n-2c}(Z\times Y)_\mathbb{Q}$$
 such that
 $$[ \mathcal{Z}'_{\rm Lef}]_*:  H^{n-2c}(Z,\mathbb{Q})\rightarrow H^{n-2c}(Y,\mathbb{Q})$$
 is surjective. In other words, the Lefschetz standard conjecture predicts that
 $$ H^{n-2c}(Y,\mathbb{Q})=\widetilde{N}_{c,{\rm cyl}}H^{n-2c}(Y,\mathbb{Q})$$
 for $Y$ smooth projective of dimension $n-c$.
 Coming back to $X$, any class $\alpha$ in $\widetilde{N}^cH^n(X,\mathbb{Q})$ is of the form
 $\Gamma_*\beta$ for some class $\beta\in H^{n-2c}(Y,\mathbb{Q})$, for some smooth (nonnecessarily connected) projective variety $Y$ of dimension $n-c$, and the previous construction shows that, assuming the Lefschetz standard conjecture for $Y$,
 one has
 $$\alpha=[\Gamma\circ \mathcal{Z}'_{\rm Lef}
 ]_*\gamma$$
 for some $\gamma\in H^{n-2c}(Z,\mathbb{Q})$, where $Z$ is constructed as above. As $\Gamma\circ \mathcal{Z}'_{\rm Lef}\in{\rm CH}^{n-c}(Z\times X)_\mathbb{Q}$,
 where $Z$ is smooth projective of dimension ${\rm dim}\, n-2c$, this proves the equality (\ref{eqagalef}).

 Coming back to $\mathbb{Z}$-coefficients, there is one case where $\widetilde{N}_{{\rm cyl}}H^{k}(X,\mathbb{Z})$ and $\widetilde{N}H^{k}(X,\mathbb{Z})$ exactly compare,
 namely
 \begin{prop}\label{procylegalstrong}  We have, for any $c$ and any smooth projective variety $X$ of dimension $n$,
  \begin{eqnarray}\label{eqegalcystr} \widetilde{N}_{n-c,{\rm cyl}}H^{2c-1}(X,\mathbb{Z})=\widetilde{N}^{c-1}H^{2c-1}(X,\mathbb{Z})
  \end{eqnarray}
 \end{prop}
 \begin{proof} The inclusion $\subset$ follows from Lemma \ref{lequisert}.  For the reverse inclusion,    $\widetilde{N}^{c-1}H^{2c-1}(X,\mathbb{Z})$ is by definition generated by the groups $ \Gamma_* H^1(Y,\mathbb{Z})$, for all
 smooth projective $Y$ of dimension $n-c+1$ and all correspondences $\Gamma\in{\rm CH}^{n}(Y\times X)$.
 For each such $Y$, there exists a Poincar\'{e} (or universal) divisor
 $$\mathcal{D}\in{\rm CH}^1({\rm Pic}^0(Y)\times Y)$$
 such that
 $$[\mathcal{D}]_*:H_1({\rm Pic}^0(Y),\mathbb{Z})\rightarrow H^1(Y,\mathbb{Z})$$
 is the natural isomorphism. (We identify here ${\rm Pic}^0(Y)$ with the intermediate Jacobian $J^1(Y)=H^{0,1}(Y)/H^1(Y,\mathbb{Z})$ via the Abel map.)  Let now
 $$\mathcal{Z}:=(Id,\Gamma)_*\mathcal{D}\in{\rm CH}^{c}({\rm Pic}^0(Y)\times X).$$
 We have
 $$[\mathcal{Z}]_*=[\Gamma]_*\circ [\mathcal{D}]_*: H_1({\rm Pic}^0(Y),\mathbb{Z})\rightarrow H^{2c-1}(X,\mathbb{Z})$$
and it  has the same image as $[\Gamma]_*$. Thus we proved that  $\widetilde{N}^{c-1}H^{2c-1}(X,\mathbb{Z})$ is generated by cylinder homomorphisms associated to families
 of cycles in $X$ of dimension $n-c$ parameterized by  a smooth basis.
 \end{proof}

 Note that for $c=n-1$, Proposition \ref{procylegalstrong} applies to degree $2n-3$ cohomology, that is, degree $3$ homology, which we will be considering in next section.

The niveau $1$ of  the cylinder filtration  produces stable birational invariants. The following result strengthens the corresponding statement for strong coniveau in
 \cite{BeOt}.
 \begin{prop}\label{propstablebirat} The quotient $H_k(X,\mathbb{Z})/\widetilde{N}_{1,{\rm cyl}}H_k(X,\mathbb{Z})$
 is a stable birational invariant of a smooth projective variety $X$.
  \end{prop}
 \begin{proof} The invariance under the relation $X\sim X\times \mathbb{P}^r$ is obvious by the projective bundle formula which shows that
 $H_k(X\times \mathbb{P}^r,\mathbb{Z})= H_k(X,\mathbb{Z})+\widetilde{N}_{1,{\rm cyl}}H_k(X\times \mathbb{P}^r,\mathbb{Z})$, so that ${\rm pr}_{X*}:H_k(X\times \mathbb{P}^r,\mathbb{Z})\rightarrow  H_k(X,\mathbb{Z})$ is an isomorphism modulo $\widetilde{N}_{1,{\rm cyl}}$.
 It remains to prove the invariance under birational maps. In fact, it suffices to prove the invariance under blow-ups along smooth centers, because the considered groups admit both contravariant functorialities under pull-backs and covariant functoriality
 under proper push-forwards for  generically finite maps (see \cite[Lemma 1.9]{voisingargnano}).
For a blow-up
 $\tau:\widetilde{X}\rightarrow X$
 the standard formulas show that
 $H_k(\widetilde{X},\mathbb{Z})=\tau^* H_k(X,\mathbb{Z})+\widetilde{N}_{1,{\rm cyl}}H_k(\widetilde{X},\mathbb{Z})$, so that $\tau_*:H_k(\widetilde{X},\mathbb{Z})\rightarrow  H_k(X,\mathbb{Z})$  is an isomorphism modulo $\widetilde{N}_{1,{\rm cyl}}$.
 \end{proof}
 The following result
 is a motivation for introducing  Definition \ref{deficyl}.
 \begin{prop}\label{propcritdecdiag}  Let $X$ be a smooth projective variety admitting a cohomological decomposition of the diagonal. Then
 for any $k$ such that $2n>k>0$,
 $$\widetilde{N}_{\rm 1,{\rm cyl}}H^k(X,\mathbb{Z})=H^k(X,\mathbb{Z})=\widetilde{N}^1H^k(X,\mathbb{Z}).$$ In particular, these equalities hold
 if $X$ is stably rational.
 \end{prop}
 \begin{proof} The  second equality already appears in \cite{BeOt}. Both equalities  follow from \cite{voisinJEMS}, where the following result is proved.
 \begin{theo} \label{theojems} If a smooth projective variety $X$ of dimension $n$  admits a cohomological decomposition of the diagonal, there exist smooth projective varieties
 $Z_i$ of dimension $n-2$, integers $n_i$, and correspondences
 $\Gamma_i\in{\rm CH}^{n-1}(Z_i\times X)$ such that, choosing a point $x\in X$,
 \begin{eqnarray}\label{eqdecavecZi}
 [\Delta_X-x\times X-X\times x] =\sum_i n_i(\Gamma_i,\Gamma_i)_*[\Delta_{Z_i}]\,\,{\rm in}\,\,{H}^{2n}(X\times X,\mathbb{Z}).
 \end{eqnarray}
 \end{theo}
 In (\ref{eqdecavecZi}), the correspondence $(\Gamma_i,\Gamma_i)$ between $Z_i\times Z_i$ and $X\times X$
 is defined as ${\rm pr}_1^*\Gamma_i\cdot {\rm pr}_2^*\Gamma_i$, where we identify
 $Z_i\times Z_i \times X\times X$ with $Z_i\times X\times Z_i\times X$, which defines  the two projections
 $${\rm pr}_1,\,\,\,{\rm pr}_2: Z_i\times Z_i\times X\times X\rightarrow Z_i\times X.$$
 Another way to
 formulate (\ref{eqdecavecZi}) is obtained by introducing the transpose $ {^t\Gamma_i}\in{\rm CH}^{n-1}(X\times Z_i)$, which satisfies
 $ {^t\Gamma_i}_*=\Gamma_i^*$. Then
 (\ref{eqdecavecZi}) is equivalent to the equality of cohomological self-correspondences of $X$
 \begin{eqnarray}
 \label{eqavectGammai}
 [\Delta_X-X\times x-x\times X]=\sum_i n_i [ \Gamma_i\circ{^t\Gamma_i}]\,\, {\rm in}\,\,H^{2n}(X\times X,\mathbb{Z}).
 \end{eqnarray}
 Applying both sides of (\ref{eqavectGammai}) to any $\alpha\in H^{0< *<2n}(X,\mathbb{Z})$, we get
 $$\alpha=\sum_in_i [{\Gamma_i}]_*\circ [\Gamma_i]^*\alpha\,\, {\rm in}\,\,H^{*}(X,\mathbb{Z}),$$
 With $[\Gamma_i]^*\alpha\in H^{*-2}(Z_i,\mathbb{Z})$. As ${\rm dim}\,Z_i=n-2$ and ${\rm dim}\,\Gamma_i=n-1$, this
  proves that $\alpha\in \widetilde{N}_{1,{\rm cyl}}H^*(X,\mathbb{Z})$ and $\alpha\in \widetilde{N}^{1}H^*(X,\mathbb{Z})$.
 \end{proof}

 \begin{rema}{\rm Although Theorem \ref{theojems} is stated in \cite{voisinJEMS} only in the cohomological setting, it is true as well, with the same proof, in the Chow setting, see \cite{shen}. The same proof as above thus gives the following result.}
 \end{rema}
 \begin{theo}\label{theopasutileici}  If $X$ admits a Chow decomposition of the diagonal, there exist
 correspondences
 $\Gamma_i\in {\rm CH}^{n-1}(Z_i\times X)$ and integers $n_i$, such that
 $$\Gamma_i^*: {\rm CH}^{n>*>0}(X)\rightarrow \oplus_i{\rm CH}^{*-1}(Z_i)$$
 has for left inverse
 $\sum_i n_i {\Gamma}_{i*}$. In particular  $\sum_i n_i {\Gamma}_{i*}:\oplus {\rm CH}^{*}(Z_i)\rightarrow {\rm CH}^{*+1}(X)$ is surjective for $n-2\geq *\geq 0$.
 \end{theo}
 \begin{coro} If $X$ admits a Chow decomposition of the diagonal, the Chow groups
 ${\rm CH}^i(X)$ for $0< i<n$ satisfy
 $$ \widetilde{N}_{1,{\rm cyl}}{\rm CH}^i(X)={\rm CH}^i(X)=\widetilde{N}^1{\rm CH}^i(X),$$
 where the definition of strong coniveau and cylinder niveau is extended to Chow groups in the obvious way.
 \end{coro}
 Proposition \ref{propcritdecdiag} works as well with $\mathbb{Q}$-coefficients, so we get in this case:
 \begin{prop}\label{provariantQ} Let $X$ be a smooth projective variety admitting a cohomological decomposition of the diagonal with rational coefficients. Then
 for any $k$ such that $2n>k>0$,
 $$\widetilde{N}_{\rm 1,{\rm cyl}}H^k(X,\mathbb{Q})=H^k(X,\mathbb{Q})=\widetilde{N}^1H^k(X,\mathbb{Q}).$$ In particular, these equalities hold
 if $X$ is  rationally connected.
 \end{prop}

We now introduce a weaker notion, namely the cylinder homomorphism
filtration.  Let $X$ be a smooth projective manifold of dimension
$n$. We have $H_k(X,\mathbb{Z})\cong H^{2n-k}(X,\mathbb{Z})$ by Poincar\'e duality.

\begin{Defi} \label{deficylhomo} We define $N_{c,{\rm cyl}} H^{k}(X,\mathbb{Z})$ as the group
generated by the cylinder homomorphisms
$$ f_*\circ p^*: H_{2n-k-2c}(Z,\mathbb{Z})\rightarrow H_{2n-k}(X,\mathbb{Z})\cong H^{k}(X,\mathbb{Z})
,$$
for all morphisms $f:Y\rightarrow X$, and flat projective morphisms
$p:Y\rightarrow Z$ of relative dimension $c$, where ${\rm dim}\,Z\leq 2n-k-2c$.
\end{Defi}
In this definition,  the morphism $p^*: H_{2n-k-2c}(Z,\mathbb{Z})\rightarrow H_{2n-k}(Y,\mathbb{Z})$ is obtained at the chain level by taking the inverse image
$p^{-1}$ under the flat map $p$.
Note that we do not ask here smoothness of $Z$, and this is the main   difference with Definition \ref{deficyl}.
It is obvious that  $$N_{c,{\rm cyl}}H^{k}(X,\mathbb{Z})\subset N^{k+c-n}H^{k}(X,\mathbb{Z})$$ because
with the above notation, one has ${\rm dim}\,Y\leq 2n-k-c$.
Restricting to the case  where  $Z$ is smooth,
we claim  that
$$ \widetilde{N}_{c,{\rm cyl}}H^{k}(X,\mathbb{Z}) \subset   N_{c,{\rm cyl}} H^{k}(X,\mathbb{Z}).$$
Indeed,  $\widetilde{N}_{c,{\rm cyl}}H^{k}(X,\mathbb{Z})$ is generated by images of correspondences
$ f_*\circ p^*: H^{k-2c}(Z,\mathbb{Z})\rightarrow H^{k}(X,\mathbb{Z})$
for all morphisms $f:Y\rightarrow Z$, where $Y$ is smooth and  projective, and  morphisms
$p:Y\rightarrow Z$ of relative dimension $c$, where ${\rm dim}\,Z=n-2c$. By flattening,
there exists a commutative diagram
 $$\begin{matrix}
 & Y'&\stackrel{\tau_Y}{\rightarrow}& Y
 \\
&p'\downarrow& &p\downarrow
\\
&Z'&\stackrel{\tau_Z}{\rightarrow} &Z
\end{matrix}
,$$
where $\tau_Y: Y'\rightarrow Y$ is proper birational, $Z'$ is smooth and $\tau_Z: Z'\rightarrow Z$ is proper  birational, and $p':Y'\rightarrow Z'$ is flat.
Then we have, denoting $f':=f\circ \tau_Y$
 $$f'_* \circ { p'}^*= f_* \circ { p}^*\circ \tau_{Z*}: H_{2n-2c-k}(Z',\mathbb{Z})\rightarrow  H_{2n-k}(X,\mathbb{Z}).$$
The map $\tau_{Z*}: H_{2n-2c-k}(Z',\mathbb{Z})\rightarrow H_{2n-2c-k}(Z,\mathbb{Z})=H^{k-2c}(Z,\mathbb{Z})$  is surjective since $Z$ is smooth and $\tau_Z$ is proper birational, hence we conclude that ${\rm Im}\,f_*\circ p^*\subset .{\rm Im}\,f'_* \circ { p'}^*$, proving the claim.

In conclusion we
have the chain of inclusions
\begin{eqnarray}\label{eqchainofincl} \widetilde{N}_{c,{\rm cyl}}H^{k}(X,\mathbb{Z}) \subset   N_{c,{\rm cyl}} H^{k}(X,\mathbb{Z}) \subset N^{k+c-n}H^k(X,\mathbb{Z}).
\end{eqnarray}

We are concerned in the paper with the niveau $1$ of the cylinder  filtration, which is parameterized by curves. In this case, we can use the
following variant of the cylinder homomorphism filtration. It has the advantage that we can apply to it the
beautiful results we know about the deformation theory  of morphisms from semistable curves (see  \cite{Komimo}), while the local study of the  Hilbert scheme, even for curves on threefolds, is hard.
\begin{Defi}\label{deficylstable} We define $N_{1,{\rm cyl,st}} H^{k}(X,\mathbb{Z})$ as the group
generated by the cylinder homomorphisms
$$ f_*\circ p^*: H_{2n-k-2}(Z,\mathbb{Z})\rightarrow H_{2n-k}(X,\mathbb{Z})\cong H^{k}(X,\mathbb{Z})
,$$
for all  morphisms $f:Y\rightarrow X$,  and projective flat semi-stable  morphisms
$p:Y\rightarrow Z$ of relative dimension $1$, where ${\rm dim}\,Z\leq 2n-k-2$.
\end{Defi}
The relationships between the definitions \ref{deficylhomo} and \ref{deficylstable} is not straightforward, since
semi-stable reduction of a general flat morphism $f:Y\rightarrow Z$ of relative dimension $1$ will not exist
on $Z$ but after base change, which will change the homology of  $Z$. One may expect however that the two definitions coincide.

 We conclude this section with the case of the  smooth Fano complete intersections
 $$X=\cap_{i=1}^{N-n} Y_i\subset \mathbb{P}^N,$$ with ${\rm deg}\,Y_i=d_i$ and $\sum_i d_i\leq N$. Given such a smooth $n$-dimensional variety $X$, let $F(X)\subset G(2,N+1)$ be its Fano variety of lines.
 Being the zero-locus of a general section of a globally generated vector bundle on the Grassmannian of lines  $G(2,N+1)$, $F(X)$ is smooth
 for general $X$.
 The universal family of  lines
 $$\begin{matrix}
 & P&\stackrel{q}{\longrightarrow}& X
 \\
&p\downarrow& &
\\
&F(X)& &
\end{matrix}
$$
 provides a ``cylinder homomorphism''
 \begin{eqnarray}
 \label{eqcylinder} P_*=q_*\circ p^*:H_{n-2}(F(X),\mathbb{Z})\rightarrow H_n(X,\mathbb{Z})=H^n(X,\mathbb{Z}).
 \end{eqnarray}
 When $F(X)$ is smooth, we can choose a dimension $n-2$ smooth complete intersection $Z\stackrel{j}{\hookrightarrow}   F(X)$ of ample hypersurfaces. Then by Lefschetz theorem on hyperplane sections,
 $$j_*:H_{n-2}(Z,\mathbb{Z})\rightarrow H_{n-2}(F(X),\mathbb{Z})$$
 is surjective and thus ${\rm Im}\,P_*={\rm Im}\,P_*\circ j_*$. By smoothness
 of $Z$, we can write $(P\circ j)_*$ in cohomology
 $$(P\circ j)_*:H^{n-2}(Z,\mathbb{Z})\rightarrow H^{n}(X,\mathbb{Z}).$$
 It is then clear that ${\rm Im}\,P_*$ is contained in $\widetilde{N}_{1,{\rm cyl}}H^n(X,\mathbb{Z})$.

 \begin{theo}\label{theoci}  (i)  For any  smooth Fano complete intersection $X\subset \mathbb{P}^N$ of dimension $n$ of hypersurfaces of degrees
 $d_1,\ldots, d_{N-n}$, the morphism
 $P_*$ of (\ref{eqcylinder}) is surjective.

 (ii) We have  $N_{1,{\rm cyl},{\rm st}}H^n(X,\mathbb{Z})=H^n(X,\mathbb{Z})$.

 (iii) If either  $F(X)$ has the expected dimension $2N-2-\sum_i(d_1+1)$ and   ${\rm Sing}\,F(X)$ is of codimension $\geq n-2$ in $F(X)$,
 or ${\rm dim}\,X=3$, we have
  $$H^n(X,\mathbb{Z})=\widetilde{N}_{1,{\rm cyl}}H^n(X,\mathbb{Z})=\widetilde{N}^1H^n(X,\mathbb{Z}).$$
 \end{theo}
Note that (ii) is not directly implied by (i) when $F(X)$ is singular, because ${\rm dim}\,F(X)$ can be $> n-2$ and we cannot apply Lefschetz
hard section theorem to reduce to a $Z\subset F(X)$ of dimension $n-2$.

 \begin{proof}[Proof of Theorem \ref{theoci}]   (i)  We first prove
 \begin{claim}\label{claimtardif} It suffices to prove the surjectivity statement of  (i)
 for a general smooth $X$  for which the variety of lines $F(X)$ is smooth (or equivalently any such $X$).
 \end{claim}
 \begin{proof}
 Indeed, let $X_0$ be a smooth complete intersection as above and choose a family $\mathcal{X}\rightarrow \Delta$ of smooth deformations
 $X_t$ of $X_0$ parameterized by the disk, so that the
 general  fiber $\mathcal{X}_t$ has its  variety of lines $F(\mathcal{X}_t)$  smooth of the expected dimension.
 Then we can consider the corresponding family
 $\mathcal{F}\rightarrow \Delta$ of Fano varieties of lines, and we have the family of
 cylinder homomorphisms
 $$P_*:H_{n-2}(\mathcal{F}_t,\mathbb{Z})\rightarrow H_{n}(\mathcal{X}_t,\mathbb{Z}).$$
 Now we observe that we can assume that we have a topological retraction $r_{\mathcal{F}}:\mathcal{F}\rightarrow \mathcal{F}_0$, compatible via $P$ with
 a topological retraction $r_{\mathcal{X}}:\mathcal{X}\rightarrow \mathcal{X}_0$. By smoothness,
 $r_{X}$ induces a homeomorphism $\mathcal{X}_t\cong \mathcal{X}_0$, hence an isomorphism
 $${r_{\mathcal{X}}}_{*}:H_n(\mathcal{X}_t,\mathbb{Z})\cong H_n(\mathcal{X}_0,\mathbb{Z}).$$
 As we have
 $${r_{\mathcal{X}}}_{*}\circ P_* =P_*\circ {r_{\mathcal{F}}}_{*}:H_{n-2}(\mathcal{F}_t,\mathbb{Z})\rightarrow H_n(\mathcal{X}_0,\mathbb{Z}),$$
 we see that the surjectivity of
 $P_*:H_{n-2}(\mathcal{F}_t,\mathbb{Z})\rightarrow H_n(\mathcal{X}_t,\mathbb{Z})$ implies the surjectivity of
 $P_*:H_{n-2}(\mathcal{F}_0,\mathbb{Z})\rightarrow H_n(\mathcal{X}_0,\mathbb{Z})$.
\end{proof}
  The claim being proved, we now assume that $F(X)$ is smooth and we   show  that
  $P_*: H_{n-2}(F(X),\mathbb{Z})\rightarrow H_n(X,\mathbb{Z})$ is surjective. We now  claim
   that it suffices to prove that
  the primitive homology of $X$ is in the image of $P_*$. If $n$ is odd, the homology and primitive homology coincide so there is nothing to prove. If
  $n=2m$, we observe that some special $X$ which are smooth and with  variety of lines smooth of the expected dimension,  contain $m$-cycles $W$ which are of degree $1$ and whose class
   is in ${\rm Im}\,P_*$. For example, we choose $X$ to have a $m$-dimensional linear sections which are  the
   union of two cones over complete intersections in $\mathbb{P}^{N-m-1}$. Each component of the cone has its class contained in ${\rm Im}\,P_*$ so it suffices that the various degrees are  coprime. The class
  $[W]\in H_n(X,\mathbb{Z})$ then maps via $j_*$ to the generator of $H_n(\mathbb{P}^N,\mathbb{Z})$, where $j$ is  the inclusion
  map of  $X$ in $\mathbb{P}^N$,   and by definition  ${\rm Ker}\,j_*=:H_n(X,\mathbb{Z})_{\rm prim}$. It is clear that $[\mathbb{P}^m]$ is in the image of $P_*$, so if the image of $P_*$ contains ${\rm Ker}\,j_*$, it contains the whole of $H_n(X,\mathbb{Z})$, which proves the claim.

 We next  restrict  as above the cylinder homomorphism to a
  smooth  $Z\subset F(X)$ of dimension $n-2$.
We will now show  that the image of $P_{Z*}: H_{n-2}(Z,\mathbb{Z})\rightarrow H_n(X,\mathbb{Z})$ contains $H_n(X,\mathbb{Z})_{\rm prim}$.
By the theory of vanishing cycles \cite[2.1]{voisinbook}, it suffices to show that ${\rm Im}\,P_{Z*}$ contains one vanishing cycle,
since they are all conjugate and generate $H_n(X,\mathbb{Z})_{\rm prim}$.
  Let $Y\subset \mathbb{P}^{N+1}$ be  a general  smooth complete intersection of hypersurfaces of degrees
  $d_1,\ldots, \,d_{N-n}$, so that ${\rm dim}\,Y=n+1$, $Y$ is smooth,  $F(Y)$ is smooth and $Y$ is covered by lines. We choose
  a general complete intersection $Z_Y$ of ample hypersurfaces  $Z_Y\subset F(Y)$
   with the following properties: one has  ${\rm dim}\,Z_Y=n$,
   the restricted  family of lines  gives a dominating (generically finite) morphism $q_Y: P_Y\rightarrow Y$, and, letting $X\subset Y$ be a
   general hyperplane section, $F(X)$ is smooth of the expected dimension, and $Z_Y\cap F(X)=:Z$ is a smooth complete intersection in $F(X)$ as above.
   As $X$ is chosen to be a general hyperplane section of  $Y$,
   $X':=q_Y^{-1}(X)\subset P_Y$ is by Bertini a smooth hypersurface $X'\subset P_Y$.
   Furthermore the image of $q_{Y,X*}: H_n(X',\mathbb{Z})\rightarrow H_n(X,\mathbb{Z})$ contains a vanishing cycle since when $X$ has a nodal degeneration at a generic point $y$ of $Y$, $X'$ also acquiers a nodal degeneration at all the preimages of $y$ in $P_Y$, assuming $q_Y$ is \'etale   over a neighborhood of $y$. (This argument appears in \cite[p 2.14]{bloch}.)
   Finally, we observe that,  via
   $p_Y:P_Y\rightarrow Z_Y$, $X'$  identifies naturally with the blow-up of $Z_Y$ along $Z$ so that
   $H_n(X',\mathbb{Z})=H_{n}(Z_Y,\mathbb{Z})\oplus H_{n-2}(Z,\mathbb{Z})$, and  that the image of the map
   $P_{Y*}: H_{n}(Z_Y,\mathbb{Z})\rightarrow H_{n}(X,\mathbb{Z})$ is contained in  the image of the restriction map $H_{n+2}(Y,\mathbb{Z})\rightarrow H_n(X,\mathbb{Z})$
   which is equal  to $\mathbb{Z}h^m$ by the Lefschetz theorem on hyperplane sections. The fact that the image
   of  $q_{Y,X*}: H_n(X',\mathbb{Z})\rightarrow H_n(X,\mathbb{Z})$ contains a vanishing cycle thus implies that the image of
   $P_{Z*}:H_{n-2}(Z,\mathbb{Z})\rightarrow H_n(X,\mathbb{Z})$ contains a vanishing cycle.
  Thus (i) is proved.

   \vspace{0.5cm}

   (ii) We modify the construction above as follows : first of all we replace
   $\mathcal{F}\rightarrow \Delta$ by a family $\mathcal{Z}\subset \mathcal{F}$ whose fiber  over $t\in \Delta^*$ is
   a $n-2$-dimensional complete intersection $\mathcal{Z}_t\subset F(\mathcal{X}_t)$ of ample hypersurfaces. Finally, we replace $\mathcal{Z}$ by
 the union   $\mathcal{Z}'$ of irreducible components of $\mathcal{Z}$ which dominate $\Delta$. Then the central fiber $\mathcal{Z}'_0$ has dimension
 $n-2$, and for the general fiber, we know by (i), by smoothness of  $F(\mathcal{X}_t)$,  and by the lefschetz theorem on hyperplane sections that
 the restriction ${P}'$ of ${P}$ to $\mathcal{Z}'$ has the property that  $P'_{t*}:H_{n-2}(\mathcal{Z}'_t,\mathbb{Z})\rightarrow H_{n}(\mathcal{X}_t,\mathbb{Z})$
 is surjective. We then conclude as in the proof of Claim \ref{claimtardif} that $P'_{0*}:H_{n-2}(\mathcal{Z}'_0,\mathbb{Z})\rightarrow H_{n}(\mathcal{X}_0,\mathbb{Z})$ is surjective, and as ${\rm dim}\,\mathcal{Z}'_0=n-2$ and the fibers of
 $P'_0\rightarrow \mathcal{Z}'_0$ are smooth, (ii) is proved.

  \vspace{0.5cm}

  (iii)
  The case where ${\rm dim}\,X=3$ is  a consequence of (ii) and  of Theorem \ref{theocylRC}. Indeed,
  (ii) says that
  $H_3(X,\mathbb{Z})=N_{1,{\rm cyl,st}}H_3(X,\mathbb{Z})$ and
  by Theorem \ref{theocylRC}, we thus have
  $H_3(X,\mathbb{Z})=\widetilde{N}^1_{\rm cyl}H_3(X,\mathbb{Z})$, hence a fortiori $H_3(X,\mathbb{Z})=\widetilde{N}^1H_3(X,\mathbb{Z})$.

   We now  conclude the proof when
  $F(X)$ has the right dimension and ${\rm Sing}\,F(X)$ is of codimension $\geq n-2$ in $F(X)$. As the Fano variety of lines $F(X)$ has the right dimension, we know already
 by the proof of (ii) that if $Z\subset X$ is a general complete intersection of ample  hypersurfaces which is of dimension $n-2$,
 the cylinder homomorphism $[P]_*:H_{n-2}(Z,\mathbb{Z})\rightarrow H_n(X,\mathbb{Z})$ is surjective. Furthermore, the assumption on ${\rm Sing}\,F(X)$
 implies that $Z$ has isolated singularities. We now apply  Proposition \ref{proisole} proved in   section \ref{secfinal}, which says that
  ${\rm Im}\,([P]_*:H_{n-2}(Z,\mathbb{Z})\rightarrow H_n(X,\mathbb{Z}))$ is  contained in $\widetilde{N}_{1,{\rm cyl}}H^n(X,\mathbb{Z})$. Thus
  $\widetilde{N}_{1,{\rm cyl}}H^n(X,\mathbb{Z})=H^n(X,\mathbb{Z})$ and a fortiori $\widetilde{N}^{1}H^n(X,\mathbb{Z})=H^n(X,\mathbb{Z})$ by Lemma \ref{lequisert}.
  \end{proof}

 \begin{rema} {\rm Theorem \ref{theoci} (i) is proved   in  \cite{shimada} with $\mathbb{Q}$-coefficients.}
 \end{rema}

\section{Proof of Theorem \ref{theomain} \label{secproof}}
\subsection{Abel-Jacobi map for $1$-cycles \label{secaj}}
Let $X$ be a smooth complex projective manifold of dimension $n$. For any smooth connected projective curve
$C$ and cycle $\mathcal{Z}\in{\rm CH}^{n-1}(C\times X)$, one has an Abel-Jacobi map
\begin{eqnarray}\label{eqmorphintjac}\Phi_\mathcal{Z}: J(C)\rightarrow J^{2n-3}(X),
\end{eqnarray}
$$z\mapsto \Phi_X(\mathcal{Z}_*(z)),$$
where $J^{2n-3}(X)=H^{2n-3}(X,\mathbb{C})/(F^{n-1}H^{2n-3}(X,\mathbb{C})\oplus H^{2n-3}(X,\mathbb{Z})_{\rm tf})$.
The morphism $\Phi_\mathcal{Z}$ is the morphism of complex tori
associated with the morphism of Hodge structures
\begin{eqnarray}\label{eqmorphcyl}  [\mathcal{Z}]_*:H^1(C,\mathbb{Z})\rightarrow H^{2n-3}(X,\mathbb{Z})_{\rm tf}.
\end{eqnarray}

By definition, the images of all morphisms
$[\mathcal{Z}]_*$ as above generate $\widetilde{N}_{1,\rm cyl}H^{2n-3}(X,\mathbb{Z})_{\rm tf}$, and
applying Proposition \ref{procylegalstrong}, we find that they generate as well
$\widetilde{N}^{n-2}H^{2n-3}(X,\mathbb{Z})_{\rm tf}\subset N^{n-2}H^{2n-3}(X,\mathbb{Z})_{\rm tf}$.

Consider first the case of  a general smooth projective threefold. As proved in \cite{CTV},
the group $H^3(X,\mathbb{Z})/N^1H^3(X,\mathbb{Z})$ has no torsion, as it injects into the unramified cohomology group
$H^0(X_{{\rm Zar}},\mathcal{H}^3(\mathbb{Z}))$, and the sheaf $\mathcal{H}^3(\mathbb{Z})$ has no torsion. It follows that
the group $H^3(X,\mathbb{Z})_{\rm tf}/N^1H^3(X,\mathbb{Z})_{\rm tf}$ has no torsion. The inclusion of lattices
$$N^1H^3(X,\mathbb{Z})_{\rm tf} \subset H^3(X,\mathbb{Z})_{\rm tf}$$
is a morphism of integral Hodge structures of weight $3$  which, thanks to the fact that $H^3X,\mathbb{Z})_{\rm tf}/N^1H^3(X,\mathbb{Z})_{\rm tf}$ has no torsion,
 induces an injection of the corresponding intermediate Jacobians
$$J(N^1H^3(X,\mathbb{Z})_{\rm tf})\hookrightarrow J(H^3(X,\mathbb{Z})_{\rm tf})=J^3(X).$$

In higher dimension, it is observed by Walker \cite{walker} that
the Abel-Jacobi map for $1$-cycles
$$\Phi_X:{\rm CH}_1(X)_{{\rm alg}}\rightarrow  J^{2n-3}(X)$$
factors through a surjective morphism
\begin{eqnarray}\label{eqwalkerliftavril} \widetilde{\Phi}_X: {\rm CH}_1(X)_{{\rm alg}}\rightarrow J(N^{n-2}H^{2n-3}(X,\mathbb{Z})_{\rm tf})
\end{eqnarray}
where
the intermediate Jacobian $J(N^{n-2}H^{2n-3}(X,\mathbb{Z})_{\rm tf})$  is not in general
a subtorus of $J(H^{2n-3}(X,\mathbb{Z})_{\rm tf})$. The point is that it is not  necessarily the case  for higher coniveau $n-2>1$  that
$$N^{n-2}H^{2n-3}(X,\mathbb{Z})_{\rm tf}\subset H^{2n-3}(X,\mathbb{Z})_{\rm tf}$$
is a saturated sublattice. We refer to  \cite{suzuki} for the discussion  of such  phenomena. There is a related stable birational
invariant, which is the torsion of the group $$H^{2n-3}(X,\mathbb{Z})_{\rm tf}/N^{n-2}H^{2n-3}(X,\mathbb{Z})_{\rm tf}.$$
Concerning the Walker lift (\ref{eqwalkerliftavril}), Suzuki proves
\begin{theo}\cite{suzuki} \label{theosuzuki} Let $X$ be a rationally connected manifold of dimension $n$. Then the Walker Abel-Jacobi map
$\widetilde{\Phi}_X: {\rm CH}_1(X)_{{\rm alg}}\rightarrow J(N^{n-2}H^{2n-3}(X,\mathbb{Z})_{\rm tf})$ is injective on torsion.
\end{theo}
Let us come back to  a general  surjective morphism $\phi: A\rightarrow B$ of complex tori $A,\,B$ that we represent as quotients

$$A=A_{0,1}/A_{\mathbb{Z}},\,\,\, B=B_{0,1}/B_{\mathbb{Z}},$$ of complex vector spaces by lattices,  with  induced morphisms
$$\phi_\mathbb{Z}=\phi_*:A_{\mathbb{Z}}\rightarrow B_{\mathbb{Z}},$$
$$ \phi_{0,1}=\phi_*:A_{0,1}\rightarrow B_{0,1}$$
respectively on integral homology  $H_1(\,,\mathbb{Z})$ and on $H_{0,1}$-groups. The subgroup
${\rm Ker}\,\phi$ is a finite union of  translates of
the subtorus
$$K:={\rm Ker}\,\phi_{0,1}/{\rm Ker}\,\phi_{\mathbb{Z}}.$$
More precisely,
\begin{lemm} \label{letrivsurtore} Let \begin{eqnarray}
\label{eqDphi} D_\phi:=\{\alpha\in A_{\mathbb{Q}},\,\phi_\mathbb{Q}(\alpha)\in B_\mathbb{Z}\}.
\end{eqnarray} Then (i)
the group $T_\phi=D_\phi/A_{\mathbb{Z}}$ is  isomorphic to the   torsion  subgroup of
${\rm Ker}\,\phi$ and
\begin{eqnarray} \label{eqpourtoretransl} {\rm Ker}\,\phi=K+ T_\phi.
\end{eqnarray}

(ii) The group $T_\phi/{\rm Ker}\,\phi_\mathbb{Q}$ is isomorphic to the group of connected components of ${\rm Ker}\,\phi$.
 \end{lemm}

 \begin{proof} (i) A torsion point of $A$ is an element of
 $A_\mathbb{Q}/A_\mathbb{Z}$ and it is  in ${\rm Ker}\,\phi$ when any of its  lifts $\alpha$   in $A_\mathbb{Q}$ maps to
 $B_\mathbb{Z}$ via $\phi_\mathbb{Q}$. This proves the first statement.
 For  the equality (\ref{eqpourtoretransl}), as  $K\subset {\rm Ker}\,\phi$ and
 $T_\phi\subset {\rm Ker}\,\phi$,  we just have to show  that ${\rm Ker}\,\phi\subset K+ T_\phi$. The result has nothing to do with complex tori,
 as we can work as well with the corresponding real tori $A_{\mathbb{R}}/A_{\mathbb{Z}}$, $B_{\mathbb{R}}/B_{\mathbb{Z}}$ which are naturally isomorphic as real tori to $A$ and $B$ respectively.
 Let $t\in {\rm Ker}\,\phi$, and let $t_\mathbb{R}$ be a lift of $t$ in $A_\mathbb{R}$. Then
 $\phi_{\mathbb{R}}(t)\in B_{\mathbb{Z}}$. Let $b_t=\phi_{\mathbb{R}}(t)\in B_{\mathbb{Z}}$ and let
 $$K_{\mathbb{R},t}=\{v\in A_\mathbb{R},\,\phi_{\mathbb{R}}(v)=b_t\}\subset A_\mathbb{R}.$$
 Then $K_{\mathbb{R},t}$ is  affine, modeled on the vector space ${\rm Ker}\,\phi_{\mathbb{R}}$,   contains $t_\mathbb{R}$, and is defined over $\mathbb{Q}$.
 Hence it has a rational point $t_\mathbb{Q}$ which belongs to $D_\phi$ and thus
  $$t_{\mathbb{R}}=t_\mathbb{Q}+t'$$ with
  $t'\in{\rm Ker}\,\phi_{\mathbb{R}}$, which proves that $t\in K+ T_\phi$ by projection modulo $A_\mathbb{Z}$  since $K={\rm Ker}\,\phi_\mathbb{R}/{\rm Ker}\,\phi_\mathbb{Z}$.

  (ii)  We have ${\rm Tors}\,K={\rm Ker}\,\phi_\mathbb{Q}/{\rm Ker}\,\phi_\mathbb{Z}$, so $T_\phi/{\rm Ker}\,\phi_\mathbb{Q}$ is isomorphic to
  ${\rm Tors}\,({\rm Ker}\,\phi)/ {\rm Tors}\,K$. Using the fact that ${\rm Ker}\,\phi$ is a group which is  a finite union of translates of
the divisible group   $K$, it is immediate to see that  ${\rm Tors}\,({\rm Ker}\,\phi)/ {\rm Tors}\,K$ is isomorphic to the group of connected components of
  ${\rm Ker}\,\phi$.
 \end{proof}
  \begin{rema}{\rm By (ii) the group $T_\phi$  is finite if $\phi$ is an isogeny, and in general it is finite modulo the torsion points of $A$ contained in the
 connected component $K$ of $0$ of ${\rm Ker}\,\phi$.  It follows that, in the formula (\ref{eqpourtoretransl}), we can replace $T_\phi$ by a finite subgroup of $T_\phi$.
 }
 \end{rema}
We will also  use the following property  of the group $T_\phi$.
\begin{lemm}\label{leautredeftphi} Let as above $\phi: A\rightarrow B$ be a surjective morphism of   tori.
Then, with  notation as  above,  the group $T_\phi$ maps surjectively, via
$$\phi_\mathbb{Q}=\phi_*:A_{\mathbb{Q}}\rightarrow B_{\mathbb{Q}},$$
to $B_\mathbb{Z}/\phi_{\mathbb{Z}}(A_\mathbb{Z})$. The kernel of the  map
$\overline{\phi}$ so defined is ${\rm Ker}\, \phi_\mathbb{Q}/{\rm Ker}\, \phi_\mathbb{Z}$ (that is, the torsion subgroup of $K$). The image of
$\overline{\phi}$ is isomorphic to ${\rm Tors}\,(B_\mathbb{Z}/{\rm Im}\,\phi_\mathbb{Z})$. In particular, ${\rm Im}\,\overline{\phi}$
is isomorphic to the group of connected components of ${\rm Ker}\,\phi$.
\end{lemm}
\begin{proof} We have indeed by definition $T_\phi=D_\phi/A_\mathbb{Z}$, where
$D_\phi=\phi_\mathbb{Q}^{-1}(B_{\mathbb{Z}})$ by (\ref{eqDphi}).
Using the fact that $\phi_\mathbb{Q}:A_{\mathbb{Q}}\rightarrow B_{\mathbb{Q}}$ is surjective, we get that
$\phi_\mathbb{Q}: D_\phi\rightarrow B_\mathbb{Z}$ is surjective.
The kernel of the induced surjective  map
$$\overline{\phi_\mathbb{Q}}: D_\phi\rightarrow B_\mathbb{Z}/\phi_{\mathbb{Z}}(A_\mathbb{Z})$$
is clearly ${\rm Ker}\,\phi_\mathbb{Q}+A_\mathbb{Z}$, hence $\overline{\phi_\mathbb{Q}}$ factors  through $T_\phi$, and the induced map
 $\overline{\phi}:T_\phi\rightarrow B_\mathbb{Z}/\phi_{\mathbb{Z}}(A_\mathbb{Z})$ has for kernel the image of
 ${\rm Ker}\,\phi_\mathbb{Q}$ in $T_\phi$. For the last point, as $T_\phi$ is of torsion, ${\rm Im}\,\overline{\phi}$ is of torsion, and conversely,
 a torsion element of  $B_\mathbb{Z}/\phi_{\mathbb{Z}}(A_\mathbb{Z})$ lifts to an element of $A_\mathbb{Q}$.
\end{proof}

Coming back to the morphisms induced by the Abel-Jacobi map, the inclusion of the  finite index sublattice
$$\widetilde{N}_{1,\rm cyl}H^{2n-3}(X,\mathbb{Z})_{\rm tf}\rightarrow N^{n-2}H^{2n-3}(X,\mathbb{Z})_{\rm tf}$$
induces an isogeny of intermediate Jacobians
\begin{eqnarray}
\label{eqpourmapjac} J(\widetilde{N}_{1,\rm cyl}H^{2n-3}(X,\mathbb{Z})_{\rm tf})\rightarrow J(N^{n-2}H^{2n-3}(X,\mathbb{Z})_{\rm tf}).
\end{eqnarray}
By definition of $\widetilde{N}_{1,{\rm cyl}}$, for any smooth projective curve $C$ and codimension-$n-1$ cycle
$\mathcal{Z}\in {\rm CH}^{n-1}(C\times X)$,    the
morphism $[\mathcal{Z}]^*:H^1(C,\mathbb{Z})\rightarrow H^{2n-3}(X,\mathbb{Z})_{\rm tf}$ takes value in
$$\widetilde{N}_{1,{\rm cyl}}H^{2n-3}(X,\mathbb{Z})_{\rm tf}=\widetilde{N}^{n-2}H^{2n-3}(X,\mathbb{Z})_{\rm tf}\subset  N^{n-2}H^{2n-3}(X,\mathbb{Z})_{\rm tf}.$$
It follows that  the morphism $\Phi_\mathcal{Z}$ of (\ref{eqmorphintjac}), or rather its Walker lift   $\widetilde{\Phi}_\mathcal{Z}$, factors through
a morphism
\begin{eqnarray}
\label{eqliftAJ}\widetilde{\widetilde{\Phi}}_\mathcal{Z}: J(C)\rightarrow J(\widetilde{N}^{n-2}H^{2n-3}(X,\mathbb{Z})_{\rm tf}).
\end{eqnarray}
Let us clarify one point. One could naively  believe that these liftings provide a further lift  of the Walker  Abel-Jacobi map
\begin{eqnarray}\label{eqabeljacgen} \widetilde{\Phi}_X: {\rm CH}^{n-1}(X)_{\rm alg} \rightarrow J(N^{n-2}H^{2n-3}(X,\mathbb{Z})_{\rm tf})
\end{eqnarray}
defined on cycles algebraically equivalent to $0$, to  a morphism
 \begin{eqnarray}\label{eqabeljacgenliftlift}\widetilde{\widetilde{\Phi}}_X: {\rm CH}^{n-1}(X)_{\rm alg} \rightarrow J(\widetilde{N}_{1,\rm cyl}H^{2n-3}(X,\mathbb{Z})_{\rm tf})=J(\widetilde{N}^{n-2}H^{2n-3}(X,\mathbb{Z})_{\rm tf}).
 \end{eqnarray}
  For $n=3$, the existence of such a lifting would imply the equality $\widetilde{N}^1H^3(X,\mathbb{Z})_{\rm tf}=N^1H^3(X,\mathbb{Z})_{\rm tf}$ which  is the content of Theorem
\ref{theomain}  and    that we prove only for rationally connected threefolds.
Indeed, by \cite{murre}, the Abel-Jacobi map (\ref{eqabeljacgen}) is the universal regular homomorphism for codimension
$2$ cycles, so such a factorization is possible only if the natural map (\ref{eqpourmapjac}) between the  two intermediate Jacobians is an isomorphism.
The reason why the various liftings (\ref{eqliftAJ}) do not allow to construct  a lift  of (\ref{eqabeljacgen})
to a morphism  (\ref{eqabeljacgenliftlift}) is the fact that a $1$-cycle $Z\in {\rm CH}_1(X)_{\rm alg}$
does not come canonically from a family of $1$-cycles parameterized by a smooth curve $C$ as above. Two different such representations could lead to two different lifts of
$\widetilde{\Phi}_X(Z)$ in $J(\widetilde{N}_{1,{\rm cyl}}H^{2n-3}(X,\mathbb{Z})_{\rm tf})$.
A first lift allows to write $Z=\partial \Gamma_1$ for some $3$-chain supported on a smooth projective surface
$S_1$  mapping to  $X$, and a second lift  will allow to write $Z=\partial \Gamma_2$ for some $3$-chain supported on a smooth projective surface
$S_2$  mapping to  $X$. Then $\Gamma_1-\Gamma_2$ has no boundary, hence provides a priori  a homology class $\gamma$
in  $H_3(X,\mathbb{Z})\cong H^{2n-3}(X,\mathbb{Z})$ which is in $N^{n-2}H^{2n-3}(X,\mathbb{Z})$ but  is not supported on a smooth surface and  has no reason to be  in
$\widetilde{N}^{n-2}H^{2n-3}(X,\mathbb{Z})$. Due to the ambiguity of the choice, the Abel-Jacobi image of $Z$ will be well-defined  only modulo these cycles $\gamma$. Note that this argument also  explains the existence of the Walker lift.

Coming back to the case where  $X$ is a rationally connected $3$-fold,  Theorem \ref{theomain} is equivalent to the fact
that
$$\widetilde{N}^1_{\rm cyl}H^3(X,\mathbb{Z})_{\rm tf}=H^3(X,\mathbb{Z})_{\rm tf}.$$
Equivalently, for some smooth projective curve $C$, and cycle $\mathcal{Z}$ as above, the morphism
(\ref{eqmorphcyl}) is surjective.
If we consider the corresponding morphism (\ref{eqmorphintjac}) of intermediate Jacobians, its surjectivity holds
once the  morphism
(\ref{eqmorphcyl}) becomes  surjective after passing to  $\mathbb{Q}$-coefficients, and the
surjectivity of (\ref{eqmorphcyl}) is equivalent to the fact that
${\rm Ker}\,\Phi_{\mathbb{Z}}$ is connected.

\subsection{Cylinder homomorphism filtration on  degree $3$ homology \label{seccyltheo}}
Recall the  definition of the cylinder homomorphism and, for niveau $1$, stable cylinder homomorphism filtrations (Definitions \ref{deficylhomo} and \ref{deficylstable}).
 The proof of Theorem \ref{theomain} has two independent steps. The first one is   the following statement that works without any rational connectedness assumption. Here we recall that, in higher dimension,  the Abel-Jacobi map for $1$-cycles has the Walker factorization
through $$\widetilde{\Phi}_X:{\rm CH}_1(X)_{\rm alg}\rightarrow J(N^{n-2}H^{2n-3}(X,\mathbb{Z})_{\rm tf}).$$

\begin{theo} \label{theocylinder} Let $X$ be a   complex projective manifold of dimension $n$. Then, if the
Walker   Abel-Jacobi map
$\widetilde{\Phi}_X:{\rm CH}_1(X)_{\rm alg}\rightarrow J(N^{n-2}H^{2n-3}(X,\mathbb{Z})_{\rm tf})$
is injective on torsion, one has
\begin{eqnarray}
\label{eqN1N1} N_{1,{\rm cyl,st}}H^{2n-3}(X,\mathbb{Z})_{\rm tf}=N^{n-2}H^{2n-3}(X,\mathbb{Z})_{\rm tf}.
\end{eqnarray}
\end{theo}
In dimension $3$,
$N^{1}H^{3}(X,\mathbb{Z})_{\rm tf}\subset H^{3}(X,\mathbb{Z})_{\rm tf}$ has torsion free cokernel so
$J(N^{1}H^{3}(X,\mathbb{Z})_{\rm tf})\rightarrow J(H^3(X,\mathbb{Z})_{\rm tf})$ is injective and $\Phi_X=\widetilde{\Phi}_X$.
Furthermore, we can apply the following theorem
 due to Bloch (see \cite{bloch-srinivas}, \cite{murre}).
\begin{theo}\label{theobloch}  Let $X$ be a smooth projective variety over
$\mathbb{C}$. The Abel-Jacobi map
$\Phi_X:{\rm CH}^2(X)_{\rm alg}\rightarrow J^3(X)$
is injective on torsion cycles.
\end{theo}
Theorem \ref{theocylinder} thus gives in this case
\begin{coro} \label{corocylinder} (Cf. Theorem \ref{theocylintro}) Let $X$ be a   complex projective threefold. Then
\begin{eqnarray}
\label{eqN1N1bis} N_{1,{\rm cyl,st}}H^{3}(X,\mathbb{Z})_{\rm tf}=N^1H^{3}(X,\mathbb{Z})_{\rm tf}.
\end{eqnarray}
\end{coro}

For rationally connected manifolds of any dimension, we can apply Suzuki's theorem \ref{theosuzuki}.  Theorem \ref{theocylinder} thus gives in this case
 \begin{coro} \label{corocylinderRC} Let $X$ be a   rationally connected complex projective manifold of dimension $n$. Then
\begin{eqnarray}
\label{eqN1N1RC} N_{1,{\rm cyl,st}}H^{2n-3}(X,\mathbb{Z})_{\rm tf}=N^{n-2}H^{2n-3}(X,\mathbb{Z})_{\rm tf}.
\end{eqnarray}
\end{coro}

We do not know if these  statements hold true  for the whole group $H^{2n-3}(X,\mathbb{Z})$ (instead of its torsion free part).
By definition, they say that if the Abel-Jacobi map for $1$-cycles  is injective on torsion, the torsion free part
of coniveau-$n-2$,  degree-$2n-3$ cohomology   of $X$ is generated by cylinder homomorphisms
$$f_*\circ p^*: H_1(C,\mathbb{Z})\rightarrow H_{3}(X,\mathbb{Z})_{\rm tf}$$
for all diagrams
 \begin{eqnarray}
 \label{diagdiag}\begin{matrix}
 & Y&\stackrel{f}{\rightarrow}& X
 \\
&p\downarrow& &
\\
&C& &,
\end{matrix}
\end{eqnarray}
where $p$ is flat semi-stable  projective of relative dimension $1$, and $C$ is {\it any} reduced  curve (possibly singular,  and not necessarily projective).

\begin{proof}[Proof of Theorem \ref{theocylinder}] We first choose a smooth connected projective curve
$C$ and a cycle $\mathcal{Z}\in{\rm CH}^{n-1}(C\times X)$ with the property
that
\begin{eqnarray} \label{eqphizZ} [\mathcal{Z}]_*: H^1(C,\mathbb{Z})\rightarrow \widetilde{N}^{n-2}H^{2n-3}(X,\mathbb{Z})_{\rm tf}
\end{eqnarray}
is surjective. We have a lot of freedom in choosing this curve. The cycle $\mathcal{Z}$ induces
a Walker  Abel-Jacobi morphism
$\widetilde{\Phi}_{\mathcal{Z}}=\widetilde{\Phi}_X\circ \mathcal{Z}_*:J(C)\rightarrow J(N^{n-2}H^{2n-3}(X,\mathbb{Z})_{\rm tf})$ with  lift

$$\widetilde{\widetilde{\Phi}}_\mathcal{Z}: J(C)\rightarrow J(\widetilde{N}^{n-2}H^{2n-3}(X,\mathbb{Z})_{\rm tf})$$
as explained in (\ref{eqliftAJ}), which is induced by the morphism of Hodge structures (\ref{eqphizZ}).
Choosing a reference point
$0\in C$, we get  an embedding $C\rightarrow J(C)$, hence a restricted Abel-Jacobi map

$$\widetilde{\Phi}_{\mathcal{Z},C,0}:C\rightarrow J(N^{n-2}H^{2n-3}(X,\mathbb{Z})_{\rm tf})$$
 with lift
$$\widetilde{\widetilde{\Phi}}_{\mathcal{Z},C,0}: C\rightarrow J(\widetilde{N}^{n-2}H^{2n-3}(X,\mathbb{Z})_{\rm tf}).$$

\begin{lemm} \label{prochoixdeC} Choosing $C$ and $0$ in an adequate way, we  can assume
the following:

(i) Let $\alpha: J(\widetilde{N}^{n-2}H^{2n-3}(X,\mathbb{Z})_{\rm tf})\rightarrow J(N^{n-2}H^{2n-3}(X,\mathbb{Z})_{\rm tf})$ be the natural isogeny
with torsion kernel  $T_\alpha$. Then
there are points $x_i\in C$, (say with $x_0=0$) such that the set of points

$$\{\widetilde{\widetilde{\Phi}}_{\mathcal{Z},C,0}(x_i)\}\subset J(\widetilde{N}^{n-2}H^{2n-3}(X,\mathbb{Z})_{\rm tf})$$
is equal to $T_\alpha$.

\vspace{0.5cm}

(ii)  The cycles $\mathcal{Z}_{x_i}-\mathcal{Z}_{x_0}$ are of torsion in ${\rm CH}_1(X)$.
\end{lemm}
\begin{proof} (i) We first start with any curve $C_0$ and  cycle $\mathcal{Z}_0$ with surjective
$[\mathcal{Z}_0]_*$ as in (\ref{eqphizZ}). Then we will replace $C_0$ by a general complete intersection curve
$C$ in $J(C_0)$ whose image in $J(\widetilde{N}^{n-2}H^{2n-3}(X,\mathbb{Z})_{\rm tf})$ passes   through all the points in $T_\alpha$. We observe that the cycle
$\mathcal{Z}_0\in {\rm CH}^{n-1}(C_0\times X)$ induces a cycle
$\mathcal{Z}_{0,J(C_0)}\in {\rm CH}^{n-1}(J(C_0)\times X)$ with the property that
$$[\mathcal{Z}_{0,J(C_0)}]_*:H_1(J(C_0),\mathbb{Z})\rightarrow \widetilde{N}^{n-2}H^{2n-3}(X,\mathbb{Z})_{\rm tf}$$
is surjective, so we can take for $\mathcal{Z}$ the restriction to $C\times X$ of $\mathcal{Z}_{0,J(C_0)}$, and
the surjectivity
of $[\mathcal{Z}]_*:H^1(C,\mathbb{Z})\rightarrow \widetilde{N}^{n-2}H^{2n-3}(X,\mathbb{Z})_{\rm tf}$ follows from the Lefschetz theorem on hyperplane
sections which gives the surjectivity of the map
$H^1(C,\mathbb{Z})\rightarrow H_1(J(C_0),\mathbb{Z})$.

\vspace{0.5cm}

(ii) We do the same construction as above, except that we first choose torsion elements $\beta_i\in J(C_0)$ over each $\alpha_i\in J(\widetilde{N}^{n-2}H^{2n-3}(X,\mathbb{Z})_{\rm tf})$. We then ask that $C\subset J(C_0)$ passes through the points $\beta_i$ at $x_i$.

As $\beta_i$ is  a torsion point of $J(C_0)$, the $0$-cycle $\{\beta_i\}-\{0\}$ is of torsion in ${\rm CH}_0(J(C_0))$, hence  the cycle
$$\mathcal{Z}_{x_i}-\mathcal{Z}_{x_0}=\mathcal{Z}_{0,J(C_0)*}(\{\beta_i\}-\{0\})$$
is of torsion in ${\rm CH}_1(X)$, which proves (ii).
\end{proof}

We also note that we can assume the cycle $\mathcal{Z}\in {\rm CH}^{n-1}(C\times X)$ to be effective, represented by a surface mapping  to $X$ and $C$ with  smooth fibers over the points $x_i$ (and semistable fibers otherwise).This follows indeed from the definition of  $ \widetilde{N}^{n-2}H^{2n-3}(X,\mathbb{Z})_{\rm tf}$ as coming from the degree $3$ homology of a smooth projective surface  $Y$ mapping to $X$. The statement thus follows from the corresponding assertion for any universal divisor on ${\rm Pic}^0(Y)\times Y$  restricted to $C\times Y$ for an  adequate choice of curve $C\subset {\rm Pic}^0(Y)$, to which we can apply Bertini type theorems by adding ample  divisors coming from $C$ and $Y$. We assume now that we are in the situation of Lemma
\ref{prochoixdeC}.
The cycles $\mathcal{Z}_{x_i}-\mathcal{Z}_{x_0}\in{\rm CH}^{n-1}(X)_{\rm alg}$  are thus of torsion
  by Lemma \ref{prochoixdeC}, (ii), and  annihilated by $\widetilde{\Phi}_X$ since
  we have
  $$\widetilde{\widetilde{\Phi}}_\mathcal{Z}(x_i-x_0)=\beta_i,\,\,\alpha\circ \widetilde{\widetilde{\Phi}}_\mathcal{Z}=\widetilde{\Phi}_{\mathcal{Z}},$$
  and $\alpha(\beta_i)=0$ by Lemma \ref{prochoixdeC},(i).
  By assumption, the  Walker Abel-Jacobi map $\widetilde{\Phi}_X$ is injective on torsion cycles, hence the cycles  $\mathcal{Z}_{x_i}-\mathcal{Z}_{x_0}\in{\rm CH}_1(X)_{\rm alg}$ are  rationally equivalent to $0$, which means that
  there exist smooth (not necessarily connected) projective surfaces   $\Sigma_i$,  and  morphisms $$f_i:\Sigma_i\rightarrow X,\,\,
  p_i:\Sigma_i\rightarrow \mathbb{P}^1,$$ such that
  \begin{eqnarray}
  \label{eqratzerodiv}  f_{i*}(p_i^{-1}(0)-p_i^{-1}(\infty))=\mathcal{Z}_{x_i}-\mathcal{Z}_{x_0}.
  \end{eqnarray}
  Let $\gamma_i$ be a continuous path from $x_0$ to $x_i$ on $C$. We thus get a real  $3$-chain
  $$\Gamma_i=(p_X)_*\mathcal{Z}_{\gamma_i}$$ in $X$ satisfying
  $$\partial\Gamma_i=\mathcal{Z}_{x_i}-\mathcal{Z}_{x_0}.$$
  Next, let $\gamma$ be  a continuous  path from $0$ to $\infty$ on $\mathbb{C}\mathbb{P}^1$. Then
  we get a real  $3$-chain
  $\Gamma'_i=f_{i*}(p_i^{-1}(\gamma))$ in $X$ also satisfying
  $$\partial\Gamma'_i=\mathcal{Z}_{x_i}-\mathcal{Z}_{x_0}.$$

  It follows that $\Gamma_i-\Gamma'_i$ satisfies
  $\partial(\Gamma_i-\Gamma'_i)=0$, hence has a homology  class
  \begin{eqnarray}\label{eqhomoclass} \eta_i\in H_{3 }(X,\mathbb{Z}),
  \end{eqnarray}
  which belongs to $N^{n-2} H_{3 }(X,\mathbb{Z})$ since  the chains $\Gamma_i,\,\Gamma'_i$ are supported on surfaces in $X$.

  We now apply to the  isogeny
    $$\alpha: J(\widetilde{N}^{n-2}H^{2n-3}(X,\mathbb{Z})_{\rm tf})\rightarrow J(N^{n-2}H^{2n-3}(X,\mathbb{Z})_{\rm tf})$$ the results of Section \ref{secaj}.
    For the clarity of the argument, it will be more convenient to use the homology groups
    $H_{3}(X,\mathbb{Z})$ instead of the cohomology groups $H^{2n-3}(X,\mathbb{Z})$ (they are isomorphic by Poincar\'e duality).
  We thus have
  the group isomorphism

  $$\overline{\alpha}: T_{\alpha}\rightarrow N^{n-2}H_{3}(X,\mathbb{Z})_{\rm tf}/\widetilde{N}^{n-2}H_{3}(X,\mathbb{Z})_{\rm tf}$$
   discussed in Lemma \ref{leautredeftphi}.
  \begin{lemm} \label{lepour3cycle} For any $i$, the class $\eta_i$ of (\ref{eqhomoclass}) satisfies
   \begin{eqnarray}
  \label{eqintamonter}\eta_i=\overline{\alpha}(\beta_i)\,\,{\rm in}\,\, N^{n-2}H_{3}(X,\mathbb{Z})_{\rm tf}/\widetilde{N}^{n-2}H_{3}(X,\mathbb{Z})_{\rm tf}.
  \end{eqnarray}
  \end{lemm}
  \begin{rema}{\rm The construction of $\eta_i$ depends on the choice of $\gamma_i$, so it is in fact  naturally defined only modulo a class coming from
  $H_1(C,\mathbb{Z})$, hence modulo $\widetilde{N}^{n-2}H_{3}(X,\mathbb{Z})_{\rm tf}$.}
  \end{rema}
  \begin{proof}[Proof of Lemma \ref{lepour3cycle}] As we are working with the torsion free part $H_3(X,\mathbb{Z})_{\rm tf}$, which embeds
  in the complex vector space $F^{2}H^{3}(X)^*$, it suffices to check
  the result after integration of classes in $F^{2}H^{3}(X)$. These classes are represented by
  closed forms $\nu$ of type $(3,0)+(2,1)$ on $X$. When pulling-back these forms on $\Sigma_i$ via $f_i$ and push-forward to
  $\mathbb{P}^1$ via $p_{i}$ we get $0$ since there are no nonzero holomorphic forms on $\mathbb{P}^1$. We thus conclude that
  $\int_{\Gamma'_i}\nu=0$,
  hence
\begin{eqnarray}\label{eqintegrale} \int_{\eta_i}\nu=\int_{\Gamma_i}\nu,
\end{eqnarray}
for any closed form $\nu$ of type $(3,0)+(2,1)$ on $X$.

It remains to understand why (\ref{eqintegrale}) is equivalent to (\ref{eqintamonter}). In fact, consider the general case
of an isogeny
$\phi:A=A_\mathbb{R}/A_\mathbb{Z}\rightarrow B_\mathbb{R}/B_\mathbb{Z}$
of real tori, with induced morphisms
$$\phi_{\mathbb{Z}}:H_1(A,\mathbb{Z})\rightarrow H_1(B,\mathbb{Z}),$$
$$\phi_{\mathbb{R}}:H_1(A,\mathbb{R})\rightarrow H_1(B,\mathbb{R})$$
on degree $1$ homology.  Then, referring to the proof of Lemma \ref{letrivsurtore} for the notation,
the  isomorphism   $\overline{\phi}:T_\phi\rightarrow B_{\mathbb{Z}}/\phi_\mathbb{Z}(A_\mathbb{Z})$
is obtained by passing to the quotient  from   the natural map
$\phi_{\mathbb{R}}^{-1}( H_1(B,\mathbb{Z}))=:D_\phi\rightarrow H_1(B,\mathbb{Z}))$ given  by restricting $\phi_{\mathbb{R}}$ to $D_\phi$.

In our case, the map $\alpha_{\mathbb{R}}$ is induced by the cylinder map associated with the cycle $\mathcal{Z}$, and the choice of a path
$\gamma_i$ from $x_0$ to $x_i$ determines a class in $H_1(J(C),\mathbb{R})$ whose image in $J(C)$ is the Abel-Jacobi image of $x_i-x_0$.
The image of this class under $\alpha_{\mathbb{R}}$ is  the element $\int_{\Gamma_i}\in F^{2}H^{3}(X)^*\cong H^{3}(X,\mathbb{R})^*$.
Hence the equality (\ref{eqintegrale}) exactly says that $\overline{\alpha}(\beta_i)=\eta_i$ modulo torsion and ${\rm Im}\,\alpha_{\mathbb{Z}}$.
 \end{proof}

 We now conclude the proof of Theorem \ref{theocylinder}.
 We start from a smooth projective  curve $C$ and cycle $\mathcal{Z}\in{\rm CH}^{n-1}(C\times X)$ satisfying
  the properties stated in Lemma \ref{prochoixdeC} and such that
 $$[\mathcal{Z}]_*(H_1(C,\mathbb{Z}))=\widetilde{N}^{n-2}H_{3}(X,\mathbb{Z})_{\rm tf}.$$
We then get as in (\ref{eqratzerodiv})
the surfaces $\Sigma_i$ and
the morphisms
\begin{eqnarray}
\label{eqdtaa} f_i:\Sigma_i\rightarrow X,\,p_i:\Sigma_i\rightarrow \mathbb{P}^1
\end{eqnarray}
with the property that
\begin{eqnarray}
\label{eqdtaa2}f_{i*}(p_{i}^{-1}(0)-p_{i}^{-1}(\infty))=\mathcal{Z}_{x_i}-\mathcal{Z}_{x_0}.
\end{eqnarray}
Let us first explain the proof in a simplified  case. Assume that there is a single index $i=1$,
 $f_1$ is an embedding along the curves  $p_{1}^{-1}(0)$ and $p_{1}^{-1}(\infty)$ which are smooth curves in $\Sigma_1$, and we have identifications of smooth curves in $X$
 \begin{eqnarray}
 \label{eqnaive}  f_1(p_{1}^{-1}(0))=\mathcal{Z}_{x_1},\,\,f_1(p_{1}^{-1}(\infty))=\mathcal{Z}_{x_0}.
 \end{eqnarray}
 In this case, we construct the singular curve
 $C'$ as the union of $C$ and   a copy of $\mathbb{P}^1$ glued by two points to $C$, with
 $0\in\mathbb{P}^1$ identified  to $x_1\in C$, $\infty\in\mathbb{P}^1$ identified  to $x_0\in C$.
 Over $C'$, we put the family $\mathcal{Z}'\rightarrow C'$ of curves in $X$ which over $C$ is $f:\mathcal{Z}\rightarrow X,\,p:\mathcal{Z}\rightarrow C$ and over $\mathbb{P}^1$ is
 $f_1:\Sigma_1\rightarrow X,\,p_1:\Sigma_1\rightarrow\mathbb{P}^1$. They glue by assumption over the intersection points using the identifications  (\ref{eqnaive}). Flatness is
 easy to check in this case. For semi-stability, it suffices to restrict to the Zariski open set $C'_0$ of $C'$ (which contains all the singular points of $C'$)
 parameterizing semi-stable fibers.

 If we now look at the cylinder homomorphism
 $$\mathcal{Z}'_*:H_1(C'_0,\mathbb{Z})\rightarrow H_3(X,\mathbb{Z})_{\rm tf},$$
 its image contains $\mathcal{Z}_*H_1(C,\mathbb{Z})=\widetilde{N}^{n-2}H_3(X,\mathbb{Z})$
 and an extra generator over the loop in $C'_0$ made of the pathes $\gamma$ on $\mathbb{P}^1$ and $\gamma_1$ on $C$ (which we can assume to avoid the points with  non-semistable fibers).
 Lemma \ref{lepour3cycle} tells us that the image of this path under $\mathcal{Z}'_*$
 is the element   $\eta_1$ of $N^{n-2}H_3(X,\mathbb{Z})_{\rm tf}$ which, together with $\widetilde{N}^{n-2}H_3(X,\mathbb{Z})_{\rm tf}$, generates $N^{n-2}H_3(X,\mathbb{Z})_{\rm tf}$. As ${\rm Im}\,\mathcal{Z}'_*\subset N_{1,{\rm cyl},{\rm st}}H_3(X,\mathbb{Z})$,   we proved
 the theorem in this case.

\begin{rema}{\rm  The reason why the above argument does not cover the general case is the fact that rational equivalence of two curves in $X$ does  not in general take the simple form described above.
}
\end{rema}

 Let us now prove the general case. Out of the data (\ref{eqdtaa}), (\ref{eqdtaa2}), we shall construct
 a modified family over a singular curve as above. Up to now, we have not been really using the fact that we are working with $1$-cycles, but we will use it now. We fix  $i$ and prove the following:
\begin{claim}
\label{claimmod} After replacing   $X$ by $X\times \mathbb{P}^r$ and modifying the family $p:\mathcal{Z}\rightarrow C,\,f:\mathcal{Z}\rightarrow X$   by gluing
 components $\mathcal{Z}'_l\rightarrow C,\,\mathcal{Z}'_l\rightarrow X$ with trivial Abel-Jacobi map,
  we can choose the rational equivalence    relation (\ref{eqdtaa}), (\ref{eqdtaa2})  so that it  takes the following form: There exist a chain
  $C_1,\ldots,\,C_m$ of smooth  curves  with two  marked points
  $s_j,\,t_j\in C_j$, glued by $t_{j}=s_{j+1}$, and
   surfaces $\Sigma_j$, $j=1,\ldots,\, m$ with two maps
  \begin{eqnarray}\label{sigmajfjpj} f_j:\Sigma_j\rightarrow X,\,\,p_j: \Sigma_j\rightarrow C_j,
  \end{eqnarray}
  satisfying the conditions:

  (i) For each $j=1,\ldots,\,m$,   $f_j$ is an embedding, and the relation$p_j$ is flat with semistable fibers  (so (\ref{sigmajfjpj}) is a family of stable maps to $X$ parameterized by $C_j$).

  (ii) For $1\leq j\leq m-1$, the stable map $f_j:p_j^{-1}(t_j)\rightarrow X$ is isomorphic to  the stable map $f_{j+1}:p_{j+1}^{-1}(s_{j+1})\rightarrow X$.

  (iii) We have equalities of stable maps $$(f_1:p_{1}^{-1}(s_{1})\rightarrow X)=(f_{\mid \mathcal{Z}_{x_i}}:\mathcal{Z}_{x_i}\rightarrow X),\,\,(f_m:p_{m}^{-1}(t_{m})\rightarrow X)=(f_{\mid \mathcal{Z}_{x_0}}:\mathcal{Z}_{x_0}\rightarrow X).$$

  (iv)  The Abel-Jacobi map $C_j\rightarrow J^{2n-3}(X)$ is trivial for each family of curves $p_j:\Sigma_j\rightarrow C_j,\,f_j:\Sigma_j\rightarrow X$.

  Furthermore, we can choose the surfaces $\Sigma_j$ to be unions of smooth surfaces with normal crossings.
  \end{claim}

 Claim \ref{claimmod} concludes the proof of Theorem \ref{theocylinder} by the same argument as before, except that the loop  $\gamma\cup \gamma_1$ on $\mathbb{P}^1\cup C$ is replaced by the  continuous path $\gamma\cup\gamma_1$ on $C'=\cup_jC_j\cup C$  constructed as follows:
 let $C'$ be the curve which is the union $\cup_jC_j\cup C$, with the points $t_j,\,s_{j+1}$ identified for $j\leq m-1$ and the points $s_1$ identified with $x_0$, the point $t_m$ identified with $x_i$.
 We choose the continuous path $\gamma$ on $\cup_jC_j\subset C'$ to be  the union of  arbitrarily chosen  pathes from
 $s_j$ to $t_j$ on $C_j$. We thus have a closed $1$-chain
$\gamma\cup \gamma_1$ on $C'$.
  There is a  family of semistable maps
 \begin{eqnarray}
 \label{eqfamilySigmaprime} f': \Sigma'\rightarrow X,\,p': \Sigma'\rightarrow C',
 \end{eqnarray}
 constructed from Claim \ref{claimmod} by gluing the various pieces
 $$f_j:\Sigma_j\rightarrow X,\,p_j:\Sigma_j\rightarrow C_j$$
 using the identifications (ii) and (iii). Using  the fact that the Abel-Jacobi map
 associated with the families $\Sigma_j\rightarrow C_j,\,\Sigma_j\rightarrow X$ is trivial on $C_j$ (assumption (iv)), we  conclude as
   in Lemma \ref{lepour3cycle} that the element
 $$\eta_i\in N^{n-2}H_3(X,\mathbb{Z})_{\rm tf}/\widetilde{N}^{n-2}H_3(X,\mathbb{Z})_{\rm tf}$$
is the image of an element of  $N_{1,{\rm cyl, st}}H_3(X,\mathbb{Z})_{\rm tf}$, namely the image of the class of  $\gamma\cup \gamma_1$ in $H_1(C',\mathbb{Z})$   under the cylinder
 homomorphism associated to the family (\ref{eqfamilySigmaprime}).
 Doing this for every $i$, we conclude
that $N_{1,{\rm cyl, st}}H_3(X,\mathbb{Z})_{\rm tf}=N^{n-2}H_3(X,\mathbb{Z})_{\rm tf}$.
\end{proof}
 \begin{proof}[Proof of Claim \ref{claimmod}] Starting from the data   (\ref{eqdtaa}), (\ref{eqdtaa2}), where we fix $i$  and write $\Sigma_i$
 as a disjoint union of smooth connected  surfaces $\Sigma_j$ mapping to $X$ via $f_j$ and to $\mathbb{P}^1$ via $p_j$, we can choose embeddings $i_j:\Sigma_j\hookrightarrow \mathbb{P}^r$
 and then $f'_j=(f_j,i_j):\Sigma_j\rightarrow X\times \mathbb{P}^r$ is an embedding.
 We can even assume that the surfaces $f'_j(\Sigma_j)$ are disjoint.
 Next we observe that, by resolution of singularities,  for each surface $\Sigma_j$, the group of nonzero rational functions on $\Sigma_j$ is generated
 by those rational functions $\phi:\Sigma_j\dashrightarrow\mathbb{P}^1$ with the following property:
  after replacing $\Sigma_j$ by a blowing-up
 $\widetilde{\Sigma}_{j,\phi}$, $\phi$ induces a surjective  map
 $\Sigma_j\rightarrow \mathbb{P}^1$ which is a Lefschetz pencil. Furthermore,  the
 divisor of $\phi$ is   up to sign  of the form $A-B-C$,  where $A,\,B,\,C$ are smooth irreducible curves.

 The $p_j$ are given by rational funtions $\phi_j$ on $\Sigma_j$, that we factor as above
 $$\phi_j=\phi_{j1}\ldots \phi_{j2}\ldots\phi_{js_j}$$
  on $\Sigma_j$, with  corresponding  blown-up surfaces $\Sigma_{jl}:=\widetilde{\Sigma}_{j,\phi_l}$ and morphisms
   $p_{jl}$ to $\mathbb{P}^1$. We choose  disjoint  embeddings $i_{jl}$, $l=1,\ldots, \,{s_j}$ of the surfaces $\Sigma_{jl}$ in
 $\mathbb{P}^r$ and do the same trick as before.
 After performing these operations, we get surfaces $\Sigma_{jl}\stackrel{f'_{jl}}{\hookrightarrow} X\times \mathbb{P}^r$
 with morphisms $p_{jl}:\Sigma_{jl}\rightarrow \mathbb{P}^1$, satisfying condition (i).

 Let $\pi:X\times \mathbb{P}^r\rightarrow X$ be the first projection. The data above satisfy the equality of cycles
 \begin{eqnarray} \label{eqratequivdefor} \pi_*(\sum_{jl}f'_{jl*}({\rm div}\,\phi_{jl}))=\mathcal{Z}_{x_i}-\mathcal{Z}_{x_0}.
 \end{eqnarray}
 Note that the curves $\mathcal{Z}_{x_i}$ and $\mathcal{Z}_{x_0}$ can be assumed to be smooth connected.
 We can also  assume, by removing finitely many points of  $X$  and working with the complement $X^0$ if necessary,  that all the irreducible  curves in the support of
 $f_{jl}( {\rm div}\,\phi_{jl})$ map  to smooth   curves $D$  in $X^0$. Denote by $D_{\alpha,0}$, (resp.
 $D_{\beta,i}$) the curves  in ${\rm Supp}\,(\sum_{jl}f'_{jl*}({\rm div}\,\phi_{jl}))$ mapping to $\mathcal{Z}_{x_0}$   (resp.   to $\mathcal{Z}_{x_i}$) via $\pi$, and, for  any other curve   $D\subset X^0$,  by $D_{\gamma,D}$  the curves  in ${\rm Supp}\,(\sum_{jl}f'_{jl*}({\rm div}\,\phi_{jl}))$ mapping to $D$. Then it follows from (\ref{eqratequivdefor}) that
\begin{eqnarray}\label{eqpourcoefficients} \sum_\alpha n_{\alpha,0}{\rm deg}\,(D_{\alpha,0}/\mathcal{Z}_{x_0})=1\\
\nonumber \sum_\beta n_{\beta,i}{\rm deg}\,(D_{\beta,i}/\mathcal{Z}_{x_i})=-1\\
\nonumber
\sum_\gamma n_{\gamma,D}{\rm deg}\,(D_{\gamma,D}/D)=0.
\end{eqnarray}
where $n_{\alpha,0}=\pm 1$ is the multiplicity of $D_{\alpha,0}$ in the cycle $\sum_{jl}f'_{jl*}( {\rm div}\,\phi_{jl})$, and similarly for $n_{\beta,i}$ and $n_{\gamma,D}$.
Assuming $r=1$ for simplicity, each curve $D_{\alpha,0}$ is rationally equivalent   in the surface
$\mathcal{Z}_{x_0}\times \mathbb{P}^1$
to
a disjoint union of ${\rm deg}\,(D_{\alpha,0}/\mathcal{Z}_{x_0})$ sections $\mathcal{Z}_{x_0}\times t_{\alpha,0,s}$, $t_{\alpha,0,s}\in\mathbb{P}^1$, modulo vertical curves
$x\times \mathbb{P}^1$, which provides   a rational function $\psi_{\alpha,0}$ on $\mathcal{Z}_{x_0}\times \mathbb{P}^1$. Similarly for $\mathcal{Z}_{x_i}$ and $D$, providing rational functions $\psi_{\alpha,0}$ on $\mathcal{Z}_{x_0}\times \mathbb{P}^1$, resp. $\psi_{\beta,i}$ on
$\mathcal{Z}_{x_i}\times \mathbb{P}^1$ and $\psi_{\gamma,D}$ on  $D\times \mathbb{P}^1$.  Using (\ref{eqpourcoefficients}) and choosing
 another   point $t_0\in \mathbb{P}^1$, we finally have a rational function
$\psi_0$ on $\mathcal{Z}_{x_0}\times \mathbb{P}^1$, resp. $\psi_i$ on $\mathcal{Z}_{x_i}\times \mathbb{P}^1$, resp.
$\psi_D$ on $D\times \mathbb{P}^1$ for each curve $D\not=\mathcal{Z}_{x_0},\,\mathcal{Z}_{x_i}$, such that  the following equalities of $1$-cycles in $X^0\times \mathbb{P}^1$
\begin{eqnarray}\label{eqdivpsi0iD} {\rm div}\,\psi_0=\sum_{\alpha,s}n_{\alpha,0}\mathcal{Z}_{x_0}\times t_{\alpha,0,s}-\mathcal{Z}_{x_0}\times t_0
\\
\nonumber
{\rm div}\,\psi_i=\sum_{\beta,s'}n_{\beta,i}\mathcal{Z}_{x_i}\times t_{\beta,i,s'}+\mathcal{Z}_{x_i}\times t_0
\\
\nonumber
{\rm div}\,\psi_D=\sum_{\gamma,s''}n_{\gamma,D} D\times t_{\gamma,D,s''}.
\end{eqnarray}
hold modulo vertical cycles $z\times \mathbb{P}^1$,
$z\in\mathcal{Z}_0(X^0)$. Equivalently, these equalities hold in $\mathcal{Z}_1(X^{00}\times \mathbb{P}^1)$, where $X^{00}\subset X^0$ is the complement in $X$ of finitely many points.

Adding to the previous surfaces $\Sigma_{jl}\stackrel{f'_{jl}}{\rightarrow}X\times \mathbb{P}^1$ and rational functions
$\phi_{jl}$  the surfaces
$$\mathcal{Z}_{x_0}\times \mathbb{P}^1,\,\mathcal{Z}_{x_i}\times \mathbb{P}^1,\,D\times \mathbb{P}^1$$
naturally contained in $X\times \mathbb{P}^1$, with  the rational functions
$\psi_{j,0}$ and  $\psi_0$,   the surface $\mathcal{Z}_{x_i}\times \mathbb{P}^1$ with  the rational functions
$\psi_{j,i}$ and $\psi_i$, and the surfaces $D\times  \mathbb{P}^1$, with  the rational functions
$\psi_{j,D}$ and  $\psi_D$, we arrived now at a situation where we have
surfaces $\Sigma'_l\subset X^{00}\times \mathbb{P}^1$, and rational functions $\chi_l$ on $\Sigma'_l$ such that
the $1$-cycles ${\rm div}\,\chi_l$ of $X^{00}\times \mathbb{P}^1$ have the property that each irreducible curve in
$\cup_l{\rm Supp}\,({\rm div}\,\chi_l)$ appears twice with opposite multiplicities $\pm1$  in $\sum_l{\rm div}\,\chi_l$, except for
$\mathcal{Z}_{x_0}\times t_0$ and $\mathcal{Z}_{x_1}\times t_0$ which appear only once, the first one with multiplicity $1$, the second one with multiplicity $-1$.
Working now over the whole of  $X$, and taking the Zariski closure $\overline{\Sigma'_l}\subset X\times \mathbb{P}^1$ of these surfaces in $X\times \mathbb{P}^1$, we find that
the $1$-cycle $\sum_l{\rm div}\,\chi_l$ of $X\times \mathbb{P}^1$  is the sum of   a  vertical $1$-cycle  $z\times \mathbb{P}^1$ and
$\mathcal{Z}_{x_0}\times t_0-\mathcal{Z}_{x_i}\times t_0$. Recalling that the supports of the divisors of the original rational functions
on the  original surfaces in $X\times \mathbb{P}^1$ were normal crossing divisors,
we can arrange by looking more closely at the surfaces $D\times \mathbb{P}^1$ (and in particular by normalizing the curves $D$) that this  is still true for the supports of the divisors ${\rm div}\,\chi_l$ in $\Sigma'_l$.  The cycle
$z\times \mathbb{P}^1$ is  rationally equivalent to $0$ in
$X\times\mathbb{P}^1$, because the cycle
$$\sum_l{\rm div}\,\chi_l= \mathcal{Z}_{x_0}\times t_0-\mathcal{Z}_{x_i}\times t_0+z\times \mathbb{P}^1$$
is rationally equivalent to $0$, and the cycle
$\mathcal{Z}_{x,0}\times t_0-\mathcal{Z}_{x,i}\times t_0$ is rationally equivalent to $0$. It follows that $z$ is rationally equivalent to  $0$ in $X$.
Writing $z=\sum_\beta\epsilon_\beta x_\beta$,
 with  $x_\beta\in X$, $\epsilon_\beta=\pm1$,
this provides us with a curve $E$ in $X$, and a rational function $\psi_E$ on $E$ with divisor
$\sum_\beta\epsilon_\beta x_\beta$, hence also a surface
$E\times \mathbb{P}^1\subset X\times \mathbb{P}^1$ with rational function $\tilde{\psi}_E$ with divisor
$\sum_\beta\epsilon_\beta x_\beta\times \mathbb{P}^1$ in $X\times \mathbb{P}^1$. As  the function $\psi_E:E\rightarrow \mathbb{P}^1$
has nonreduced fibers (corresponding  to ramification),
 this last rational function  $\tilde{\psi}_E$ does  not provide a semistable family but this is not a serious issue. Indeed, we did not ask in Claim
 \ref{claimmod} that the curves are projective, so we can simply remove the points  parameterizing nonreduced  fibers, assuming they are not gluing points.

Consider the disjoint union $\Sigma''$ of all the surfaces above, with morphisms
\begin{eqnarray}\label{eqderderpour**}p'':\Sigma''\rightarrow \mathbb{P}^1
,\,\,f'': \Sigma''\rightarrow X\times \mathbb{P}^1.
\end{eqnarray}
We find that
$$f''_*({p''}^*(0-\infty))=\mathcal{Z}_{x_0}\times t_0-\mathcal{Z}_{x_i}\times t_0$$
and more precisely that
we have
\begin{eqnarray}\label{eqderderpour*} f''({p''}^{-1}(0))=\mathcal{Z}_{x_0}\times t_0\cup A,\,\, f''({p''}^{-1}(\infty))=\mathcal{Z}_{x_i}\times t_0\cup A,
\end{eqnarray}
where both curves   $\mathcal{Z}_{x_0}\times t_0\cup A$ and $ \mathcal{Z}_{x_i}\times t_0\cup A$
have ordinary double points and both maps
$$f''_0: {p''}^{-1}(0)\rightarrow  \mathcal{Z}_{x_0}\times t_0\cup A,\,\,f''_\infty: {p''}^{-1}(\infty)\rightarrow  \mathcal{Z}_{x_i}\times t_0\cup A$$
are partial normalizations. In other words, we almost  achieved the previous situation of  (\ref{eqnaive}), but with the curve
$A$  glued to $\mathcal{Z}_{x_0}\times t_0$ and $\mathcal{Z}_{x_i}\times t_0$.
From now on, we write $X$ for $X\times \mathbb{P}^1$ and  $\mathcal{Z}_{x_0}$ for  $\mathcal{Z}_{x_0}\times t_0$,  $\mathcal{Z}_{x_i}$ for  $\mathcal{Z}_{x_i}\times t_0$.
The two fibers
\begin{eqnarray}
\label{eqstablefiber}  f''_0:{p''}^{-1}(0)\rightarrow X,\,\,{\rm  resp.}\,\, f''_\infty:{p''}^{-1}(\infty)\rightarrow X
\end{eqnarray}
 are  obtained by gluing to the curve
$\mathcal{Z}_{x_0}$, resp. $\mathcal{Z}_{x_i}$, a curve  $A_0\rightarrow X$, resp. $A_\infty \rightarrow X$, partially normalizing $A$. In other words
\begin{eqnarray}\label{eqnew2504}  {p''}^{-1}(0)=\mathcal{Z}_{x_0}\cup A_0,\,\,\,  {p''}^{-1}(\infty)=\mathcal{Z}_{x_i}\cup A_\infty.
\end{eqnarray}
Unfortunately, the two stable maps  $A_0\rightarrow X$ and  $A_\infty \rightarrow X$ a priori are different, and it
not even clear that we glue them respectively  to  $\mathcal{Z}_{x_0}$ and  $\mathcal{Z}_{x_i}$  by  the same number of points.
What we know however  is that  the genera
of  $\mathcal{Z}_{x_0}$ and $\mathcal{Z}_{x_i}$ are equal, and the genera of the fibers
$\mathcal{Z}_{x_0}\cup A_0$ and  $\mathcal{Z}_{x_i}\cup A_\infty$ appearing in (\ref{eqnew2504}) are equal, because in both cases,  by construction, these curves
are deformations of each other. It follows that the total  numbers of
gluing points  in the unions   $\mathcal{Z}_{x_0}\cup A_0$ and  $\mathcal{Z}_{x_i}\cup A_\infty$ (including those between  the components of $A_0$ or $A_\infty$)  are the same.
Let $W_0$ be the set of gluing points in ${p''}^{-1}(0)=\mathcal{Z}_{x_0}\cup_jA_j$. The set $W_0$  splits into
a union $W_0=W_{00}\sqcup W_{0A}$, where $W_{00}$ is the set of gluing points of
$\mathcal{Z}_{x_0}$ with the components $A_j$ and $W_{0A}$ is the set of gluing points between the components
$A_j$ (thus determining the curve $A_0$). We have similarly a set $W_\infty=W_{\infty i}\sqcup W_{\infty A}$.
Although
we know that  $W_0$ and $W_\infty$ have the same cardinality, we do not know
that  the sets $W_{00}$ and $W_{\infty i}$ have the same cardinality, and neither that the curves $A_0$ and $A_\infty$ have the same topology.
To circumvent this problem, we will use the following
\begin{lemm}\label{lepourrecoler}  Let $Y$ be a   complex projective manifold and let
$C,\,A_j$ be smooth curves  in $Y$  meeting transversally in distinct points
$z_1,\ldots,\,z_M$.
Then for a smooth complete intersection curve
$R$ meeting the $A_j$ and $C$ in  sufficiently many points, and for any two subsets
$\{z_{i_1},\ldots,\,z_{i_N}\}$,  $\{z_{j_1},\ldots,\,z_{j_N}\}$ of $N$ points,
there exists
a family of stable maps
\begin{eqnarray}\label{eqfamiliyporpreuve} m: \mathcal{C}\rightarrow X,\,\,\psi: \mathcal{C}\rightarrow  D
\end{eqnarray}
parameterized by a smooth connected quasiprojective curve $D$, and two points $0_1,\,0_2\in D$
with the following properties:

(i)  the stable curve
$$m_{0_1}:   \mathcal{C}_{0_1}\rightarrow X$$
over $ 0_1$ is the normalization of $C\cup R\cup_i A_i$ at the points $z_{i_1},\ldots,\,z_{i_N}$, and the stable curve
$$m_{0_2}:   \mathcal{C}_{0_2}\rightarrow X$$
over $ 0_2$ is the normalization of $C\cup R\cup_i A_i$ at the points $z_{j_1},\ldots,\,z_{j_N}$;

(ii) the family of curves (\ref{eqfamiliyporpreuve}) has trivial Abel-Jacobi map.
\end{lemm}
\begin{rema} {\rm The curve $R$ is necessary in this statement, as it adds to connectivity of the curves. Lemma \ref{lepourrecoler} is wrong without it, for topological reasons. For example, consider the union $C$ of three curves $C_1,\,C_2,\,C_3$ isomorphic to
 $\mathbb{P}^1$, and glued as follows: $C_2$ is glued to $C_1$ in two points $x,\,y$, and $C_3$ is glued to $C_1$ in two  points
 $z,\,w$. Then the curves obtained by normalizing $C$ in $x,\,y$ and  $x,\,z$  are  not  deformations of each other, since one is disconnected, not the other. }
\end{rema}
\begin{proof} [Proof of Lemma \ref{lepourrecoler}] We first choose a smooth surface $S\subset Y$, which is a complete intersection of ample hypersurfaces, and which contains
the curves $C$ and $A_i$. The curve $R$ will be any
sufficiently ample curve in $S$ meeting  $C$ and $A_i$ transversally. We choose $R$ ample enough so that
for any set $\{w_1,\ldots,w_N\} $ of $N$ points in $S$, the set of curves in
the linear system $|C+\sum_iA_i+R|$  which are  singular at all  the $w_i$'s  and have  ordinary quadratic singularities
is a Zariski open set in a projective space $\mathbb{P}_{w_{\cdot}}$ of the expected dimension.
 We next choose a curve $ \overline{D }$ in $S^{[N]} $ passing through the
two sets  $\{z_{i_1},\ldots,\,z_{i_N}\}$,  $\{z_{i_1},\ldots,\,z_{j_N}\}$ at points $0_1$, $0_2$. There exists a Zariski open set
$D\subset  \overline{D }$, and a section of the projective bundle above over $D$ passing over $0_1$ and $0_2$
through the curve  $C\cup R\cup_i A_i$. This provides a family of curves $\mathcal{C}_0$  parameterized by $D$, equipped with a choice
of $N$ singular
points, which are ordinary quadratic. The desired family is obtained by normalizing the curves of the family $\mathcal{C}_0$ at these
$N$ points.
\end{proof}

Let $n_{00}:=|W_{00}|$, $n_{\infty i}:=|W_{\infty i}|$. We can assume that $n_{00}\geq n_{\infty i}$.
Let $W$ be the set of  gluing  points in
 the stable curve $ {p''}^{-1}(0)=\mathcal{Z}_{x_0}\cup A_0$ mapping to $ X$ via $f''_0$ (see  (\ref{eqstablefiber}), (\ref{eqnew2504} )).
 Let
$N:=|W|$. Lemma \ref{lepourrecoler} says that, after gluing  a complete intersection curve $R$, we can deform  the stable map
$$f''_{0R}: \mathcal{Z}_{x_0}\cup_{W_{00}} A_0\cup R\rightarrow X, $$
obtained by gluing to $f''_0$ the inclusion of $R$,
 to any other stable map
$$f'''_{0R}: \mathcal{Z}_{x_0}\cup A'_0\cup R\rightarrow X,$$
inducing the same map on normalizations, but factoring through a different gluing of the components contained in $A_0$ or
$\mathcal{Z}_{x_0}$, assuming the total  number of identification points are the same. Furthermore, according to the lemma, this deformation can be done via a family of curves with trivial Abel-Jacobi map.
As we assumed that $n_{00}\geq n_{\infty i}$, we can  choose $A'_0=A_\infty$ glued by $n_{\infty i}$ points to $\mathcal{Z}_{x_0}$.

Looking at the proof of Lemma \ref{lepourrecoler} and using the same notation, we see that we can arrange
that the same surface $S$ also contains the curve  $\mathcal{Z}_{x_i}$ and then, because
$\mathcal{Z}_{x_0}$ and $\mathcal{Z}_{x_i}$ are algebraically equivalent, the curve $R$, which can be taken the same for
$\mathcal{Z}_{x_0}$ and $\mathcal{Z}_{x_i}$, meets $\mathcal{Z}_{x_0}$ and $\mathcal{Z}_{x_i}$ in the same number
 of points.

We now have three families of semistable curves:

(1) The original family
$\mathcal{Z}\rightarrow C,\,\mathcal{Z}\rightarrow X$ with respective  fibers $\mathcal{Z}_{x_0},\,\mathcal{Z}_{x_i}$ over $x_0,\,x_i$.

(2) The family  $f'':\Sigma''\rightarrow X,\,\,p'':\Sigma''\rightarrow \mathbb{P}^1$ of  (\ref{eqderderpour**}),
with respective  fibers   $\mathcal{Z}_{x_0}\cup A_0,\,\mathcal{Z}_{x_i}\cup A_\infty$ over $0,\,\infty$.

(3)  The family $m:\mathcal{C}\rightarrow X,\,\psi:\mathcal{C}\rightarrow D$ given by Lemma \ref{lepourrecoler} and the arguments above, with fibers
$\mathcal{Z}_{x_0}\cup A_0\cup R$,  $\mathcal{Z}_{x_0}\cup A_\infty \cup R$,

 where in (3), the family has trivial Abel-Jacobi map, and the number of attachment points of $\mathcal{Z}_{x_0}$
with $ A_\infty$  is  the same as   the number of attachment points of $\mathcal{Z}_{x_i}$
with $ A_\infty$, and furthermore the curve $R$ has the same number of points of attachment with  $\mathcal{Z}_{x_0}$ and
$\mathcal{Z}_{x_i}$.

The following lemma  will allow us (after changing $R$ if necessary)   to replace the family (1) by a family over $C$ with same Abel-Jacobi map and
 fibers $\mathcal{Z}_{x_0}\cup R\cup A_\infty,\,\mathcal{Z}_{x_i}\cup R\cup A_\infty$ and the family (2) by  a family
 parameterized by
 $ \mathbb{P}^1$,
with fibers   $\mathcal{Z}_{x_0}\cup R\cup A_0,\,\mathcal{Z}_{x_i}\cup R\cup A_\infty$.

\begin{lemm}\label{elemmefinal} Let
\begin{eqnarray}\label{eleqorigfam} f:\Sigma\rightarrow X,\,p: \Sigma\rightarrow C\end{eqnarray}
 be a family of semistable curves generically embedded in $X$
and  parameterized by
a quasiprojective smooth curve $C$. Let $x,\,y$ be two points of
$C$ and let  $S$ be a curve in $X$ meeting both curves  $f(\Sigma_x)$ and $f(\Sigma_y)$ transversally  in $M$ smooth points.
Then, up to replacing $S$ by a union $S'=S\cup S_1$, where $S_1$ is a complete intersection curve,  there exists
a family of stable maps
\begin{eqnarray}\label{eqfailynsprime}
f_{S'}:\Sigma_{S'}\rightarrow X,\,p_{S'}: \Sigma_{S'}\rightarrow C',
\end{eqnarray} where $C'\subset C$ is a Zariski dense open set of $C$ containing $x$ and $y$, with the same Abel-Jacobi map
as  (\ref{eleqorigfam}), and
with the following properties

(i) The curve $S_1$ meets $f(\Sigma_x)$ and $f(\Sigma_y)$ transversally  in $M'$ points.

(ii) The fibers of the family (\ref{eqfailynsprime}) over $x$ and $y$  are respectively
$ \Sigma_x\cup S \cup S_1$ and $\Sigma_y \cup S \cup S_1$ mapped to $X$ via
$f_x$, resp. $f_y$ on the first term.

\end{lemm}
\begin{proof}  We would like to attach the curve $S$ to the other fibers but  it does not meet a priori the other fibers, so we need to vary the curve $S$. Let us assume for simplicity that ${\rm dim}\,X=3$. We choose a  smooth surface
$T\subset  X$ containing $S$ and meeting  the curves $ f(\Sigma_x), \, f(\Sigma_y)$  (hence the general  fiber $f(\Sigma_t)$)
transversally. We  choose the curve $S_1$ in $T$ in such a way that $S\cup S_1$ has normal crossings,  is ample enough
and $S_1$ contains the  intersection points in  $f(\Sigma_x)\cap T$, and $f(\Sigma_y)\cap T$ which are not on $S$.
For a generic $t\in C$, we choose a curve
$S_t$ in the linear system $S\cup S_1$ containing the intersection $f(\Sigma_t)\cap T$, in such a way
that $S_x=S_y=S\cup S_1$. Gluing $S_t$ to $\Sigma_t$ at all the  intersection  points  $f(\Sigma_t)\cap T$ provides the desired family.
\end{proof}
Lemma \ref{elemmefinal} concludes the proof of Claim \ref{claimmod} once one observes from the proof that the same curve $S_1$
can be used for the families (1), (2), (3) above,
providing modified families of semistable curves with fibers

(1)'  $\mathcal{Z}_{x_0}\cup A_\infty\cup R\cup S_1,\,\mathcal{Z}_{x_i} \cup A_\infty\cup R\cup S_1$

(2)'  $\mathcal{Z}_{x_0}\cup A_0 \cup R\cup S_1,\,\mathcal{Z}_{x_i}\cup A_\infty\cup R\cup S_1$

(3)'  $\mathcal{Z}_{x_0}\cup A_0\cup R\cup S_1$,  $\mathcal{Z}_{x_0}\cup A_\infty \cup R\cup S_1$.

These three families, the first of which has the same Abel-Jacobi as the original family $\mathcal{Z}\rightarrow C,\,\mathcal{Z}\rightarrow X$, while the two others have trivial Abel-Jacobi map, provide the desired chain.
\end{proof}

\subsection{The case of rationally connected manifolds\label{secfinal}}
The second step in the proof of Theorem \ref{theomain} is the following statement which is valid in any dimension but concerns only
rationally connected projective manifolds.
\begin{theo}\label{theocylRC}  Let $X$ be rationally connected smooth projective of dimension $n$ over $\mathbb{C}$.
Then
\begin{eqnarray} N_{1,{\rm cyl,st}}H^{2n-3}(X,\mathbb{Z})=\widetilde{N}_{1,{\rm cyl}}H^{2n-3}(X,\mathbb{Z})=\widetilde{N}^{n-2}H^{2n-3}(X,\mathbb{Z}).
\end{eqnarray}
\end{theo}
\begin{proof} The second equality is proved in Proposition \ref{procylegalstrong}. Let $Z$ be a connected reduced  curve   with a family
$$p:Y\rightarrow Z,\,f:Y\rightarrow X,$$
of semistable curves in $X$, that is, $p$ is flat projective of relative dimension $1$ and the fibers of $p$ are semistable curves. We can also assume these curves are imbedded in $X$, so the maps are stable and automorphisms free.
It is enough  to prove that the image of
$$f_*\circ p^*: H_1(Z,\mathbb{Z})\rightarrow H_3(X,\mathbb{Z})$$
is contained in $\widetilde{N}_{1,{\rm cyl,st}}H^{2n-3}(X,\mathbb{Z})$, that is, there exists
a {\it smooth} (but not necessarily projective) variety $Z'$ and a
family of stable curves
$$p':Y'\rightarrow Z',\,f':Y'\rightarrow X,$$
with $${\rm Im}\,f_*\circ p^*\subset {\rm Im}\,( f'_*\circ {p'}^*: H_1(Z',\mathbb{Z})\rightarrow H_3(X,\mathbb{Z})).$$
Assume first  that the following holds :

(*)  {\it  At each singular point of $Z$, the
semistable map
$f_z:Y_z:=p^{-1}(z)\rightarrow X$ is stable, automorphism free,  and has unobstructed  deformations. }

Then
we take for $Y'\rightarrow Z'$  the
universal deformation of  the general fiber $f_z$, or rather, its restriction to the Zariski open set $Z'$ of the base consisting of smooth points, that is, unobstructed  stable maps, which furthermore are automorphism free.
By our assumption, there is a dense Zariski open set $Z^0\stackrel{j}{\hookrightarrow}Z$
such that  $Z\setminus Z^0$ consists of smooth points of $Z$, and
$Z^0$ maps to $ Z'$ via the classifying map $j'$. We thus  have two  commutative (in fact Cartesian) diagrams
 $$\begin{matrix}
X& \stackrel{f}\leftarrow& Y&\supseteq &Y^0&\stackrel{j''}{\rightarrow} & Y'&\stackrel{f'}{\rightarrow}& X
 \\
&&p\downarrow& &p^0\downarrow& &p'\downarrow& &
\\
&& Z&\stackrel{j}{\hookleftarrow}& Z^0& \stackrel{j'}{\rightarrow}&Z'&&
\end{matrix}
,$$
where $Y^0=p^{-1}(Z^0)$
and $$f'\circ j''=f^0,\,\,f^0:=f_{\mid Y^0}.$$
We deduce from this diagram that  the two maps
$$f'_*\circ {p'}^*:H_1(Z',\mathbb{Z})\rightarrow H_3(X,\mathbb{Z}) ,\,\, f_*\circ p^*:H_1(Z,\mathbb{Z})\rightarrow H_3(X,\mathbb{Z})$$
coincide on $H_1(Z^0,\mathbb{Z})$ which maps to both via the  maps
$$j_*:H_1(Z^0,\mathbb{Z})\rightarrow H_1(Z,\mathbb{Z}),\,j'_*:H_1(Z^0,\mathbb{Z})\rightarrow H_1(Z',\mathbb{Z})$$ induced respectively  by  the morphisms
$$ j: Z^0\hookrightarrow Z,\,\,j': Z^0\rightarrow Z'.$$
As $Z\setminus Z^0$ consists  of smooth points of
$Z$, the map
$j_*: H_1(Z^0,\mathbb{Z})\rightarrow H_1(Z,\mathbb{Z})$
is surjective, and we conclude that
$${\rm Im}\,f_*\circ p^*={\rm Im}\,f^0_*\circ {p^0}^*\subset {\rm Im}\, f'_*\circ {p'}^*,$$
 and this finishes the proof since $Z'$ is smooth.

It remains to show that we can achieve (*). This is proved in the following lemma.
First of all, we observe that after  replacing $X$ by $X\times \mathbb{P}^r$, which does not change $H_3$ since, by rational connectedness,   $H_1(X,\mathbb{Z})=0$, we can assume the map
$f_z:Y_z\rightarrow X$ to be an embedding for all $z\in Z$. In particular all maps  are stable. Then we have
\begin{lemm} Let $p: Y\rightarrow Z,\,f:Y\rightarrow X$ be a family of semi-stable curves imbedded in $X$, parameterized by a reduced curve $Z$.  There exist a Zariski open set $ Z^0\stackrel{j}{\hookrightarrow}  Z$ such that $Z\setminus Z^0$ consists of smooth points of $Z$, and a
family $\tilde{p}^0:\widetilde{Y}^0\rightarrow Z^0,\,\tilde{f}^0:\widetilde{Y}^0\rightarrow X$ of semistable curves in $X$ parameterized by $Z^0$
such that

(i) the fibers $\tilde{f}_z:\widetilde{Y}^0_z\rightarrow X$ are stable maps with unobstructed deformations;

(ii)
the cylinder map
\begin{eqnarray}\label{eqcylmappourpreuvelemme} \tilde{f}^0_*\circ (\tilde{p}^0)^*:H_1(Z^0,\mathbb{Z})\rightarrow H_3(X,\mathbb{Z})
\end{eqnarray}
coincides with the composition  $f_*\circ p^*\circ j_*$.
\end{lemm}
\begin{proof}  We first choose  a general sufficiently ample hypersurface  $W$ in $X$. There exists a Zariski open set
$Z^0_1\subset Z$, which we can assume to contain the singular points of $Z$, such that $W$ meets the fibers of $p$ only in smooth distinct points
$x_i,\,i=1,\ldots,\,N$.
Furthermore attaching to the fibers $Y_z$ a complete intersection curve $C_i$ in $X$  at each  of these intersection points, and restricting again
$Z^0_1$, we can assume (see \cite{hiha}, \cite{Komimo}) that the curves $C_i$ are smooth and disjoint, the curves
$Y_{z,1}=Y_z\cup C_z$ where $C_z:=\cup_i C_i\subset X$ are semistable and
satisfy
\begin{eqnarray}\label{eqvannormal} H^1(Y_{z},(N_{Y_{z,1}/X})_{\mid Y_z})=0.
\end{eqnarray}
The family of curves
\begin{eqnarray}\label{eqfamily1} f_1^0: Y^0_{1}\rightarrow X,\,p^0_1:Y^0_{1}\rightarrow Z^0_1
\end{eqnarray}
so constructed has the same cylinder homomorphism map (\ref{eqcylmappourpreuvelemme}) as the original family, since the part of the cylinder homomorphism
coming from the $C_i$ is easily seen to be trivial.

Unfortunately, the modified family  does not satisfy unobstructedness because the vanishing condition (\ref{eqvannormal}) is satisfied only after restriction to $Y_z$, and not on the whole of $Y_{z,1}$. We use now rational connectedness which allows us to glue  very free rational curves to the components $C_i$.
We first do this over a dense Zariski open set $M^0$
of the parameter space $M$ parameterizing the disjoint union of $N$  complete intersection curves $C_i$.
This construction modifies each curve $D=\cup_iC_i$ into an union $D'=\cup_iC'_i$ of
$C_i$ and very free rational  curves, satisfying the property that
$H^1(D',N_{D'/X})=0$.  Then we consider  the morphism
$$\eta: Z^0_1\rightarrow M,\,z\mapsto C_z$$ appearing in   the previous construction.
We can assume that $M^0$ contains $\eta({\rm Sing}\,Z^0_1)$, so
that, letting $Z^{0}:=\eta^{-1}(M^0)$, we
can construct the family
\begin{eqnarray}\label{eqtildeYZX} \tilde{p}^0: \widetilde{Y}^{0}\rightarrow Z^0,\,\tilde{f}^0:\widetilde{Y}^{0}\rightarrow X
\end{eqnarray} by gluing
to the curves $Y_{z}$ the curves $C'_i$ instead of $C_i$.  The cylinder homomorphism map for the family (\ref{eqtildeYZX}) is the same as the cylinder
homomorphism map (\ref{eqcylmappourpreuvelemme}) for the  family (\ref{eqfamily1}) because the extra part  coming from the
rational legs has its cylinder map factoring through the cylinder map associated to the family of
 curves $D'$ over $M^0$, which is trivial  since $M^0$ is smooth and  rational.
\end{proof}
 The proof of Theorem \ref{theocylRC} is now complete.
\end{proof}

We can now prove  our main theorem
\begin{theo} \label{theomainvrai} Let $X$ be a rationally connected smooth projective of dimension $n$ over $\mathbb{C}$.
Then
$N^{n-2}H^{2n-3}(X,\mathbb{Z})_{\rm tf}=\widetilde{N}^{n-2}H^{2n-3}(X,\mathbb{Z})_{\rm tf}$.
\end{theo}
 When $n=3$,  one has
 $N^{1}H^{3}(X,\mathbb{Z})=H^{3}(X,\mathbb{Z})$, so  Theorem \ref{theomain} is proved.
\begin{proof}[Proof of Theorem \ref{theomainvrai}] Let $X$ be  smooth projective rationally connected. By Corollary \ref{corocylinderRC}, one has
\begin{eqnarray}
\label{eq1fin}  N^{n-2}H^{2n-3}(X,\mathbb{Z})_{\rm tf}=N_{1,{\rm cyl,st}}H^{2n-3}(X,\mathbb{Z})_{\rm tf}.
\end{eqnarray}
By Theorem
\ref{theocylRC}, one also has
\begin{eqnarray}
\label{eq1finfin} N_{1,{\rm cyl,st}}H^{2n-3}(X,\mathbb{Z})_{\rm tf}=\widetilde{N}_{1,{\rm cyl}}H^{2n-3}(X,\mathbb{Z})_{\rm tf}.
\end{eqnarray}
Equations (\ref{eq1fin}) and (\ref{eq1finfin}) imply that
$N^{n-2}H^{2-3}(X,\mathbb{Z})_{\rm tf}= \widetilde{N}_{1,{\rm cyl}}H^{2n-3}(X,\mathbb{Z})_{\rm tf}$, where the last group is also  equal to
$\widetilde{N}^{n-2}H^{2n-3}(X,\mathbb{Z})_{\rm tf}$ by Proposition \ref{procylegalstrong}.
The result is proved.
\end{proof}

\section{Complements  and final comments \label{seccomments}}
The following important
  questions concerning the (strong or cylinder) coniveau for rationally connected manifolds remain completely open starting from dimension $4$.
As we already mentioned in the case of dimension $3$, it follows from the results of \cite{CTV} that  for a rationally connected complex projective manifold $X$, one has
$$N^1H^k(X,\mathbb{Z})=H^k(X,\mathbb{Z})$$
for any $k>0$. Indeed, the quotient $H^k(X,\mathbb{Z})/N^1H^k(X,\mathbb{Z})$ is of torsion because $X$ has a decomposition of the diagonal with
$\mathbb{Q}$-coefficients, and on the other hand, when $X$ is smooth quasiprojective,  $H^k(X,\mathbb{Z})/N^1H^k(X,\mathbb{Z})$ is  torsion free by
\cite{CTV}.
\begin{question} \label{qmain1} Let $X$ be a rationally connected complex projective manifold of dimension $n$. Is it true that
$$\widetilde{N}^{1}H^k(X,\mathbb{Z})=H^k(X,\mathbb{Z})$$
for $k>0$?
\end{question}
Of course, this  question is open only starting from $k=3$. Our main result solves this question when ${\rm dim}\,X=3$ and for the cohomology modulo torsion. In dimension $3$, it leaves open the question, also mentioned in  \cite{BeOt}, whether for a   rationally connected  threefold $X$, we have the equality  $H^3(X,\mathbb{Z})=\widetilde{N}^1H^3(X,\mathbb{Z})$.

\begin{question} \label{qmain2} Let $X$ be a rationally connected complex projective manifold of dimension $n$. Is it true that
$$N_{1,{\rm cyl}}H^k(X,\mathbb{Z})=H^k(X,\mathbb{Z})$$
for $k<2n$?
\end{question}
These  questions are not unrelated, due to the results of Section \ref{secmaterial}. For example, in degree $k=3$, a positive answer to Question
\ref{qmain1} even implies the much stronger statement  that $\widetilde{N}_{n-2,{\rm cyl}}H^3(X,\mathbb{Z})=H^3(X,\mathbb{Z})$ by Proposition \ref{procylegalstrong}.
In  degree $k=2n-2$, Question \ref{qmain2} is equivalent to asking whether $H_2(X,\mathbb{Z})$ is algebraic, a question that has been studied in
\cite{voisinharris} where it is  proved  that it would follow from the Tate conjecture on divisor classes on surfaces over a finite field.

Another   question concerns possible improvements of  Theorem \ref{theocylRC}.
\begin{question}\label{questiontechniquesurcylst}  Let $X$ be rationally connected smooth projective of dimension $n$ over $\mathbb{C}$.
Is it true that
\begin{eqnarray} N_{1,{\rm cyl}}H^{k}(X,\mathbb{Z})=N_{1,{\rm cyl,st}}H^{k}(X,\mathbb{Z})=\widetilde{N}_{1,{\rm cyl}}H^{k}(X,\mathbb{Z})
\end{eqnarray}
for any $k$?
\end{question}
We believe  that the proof of Theorem \ref{theocylRC} should work by the same smoothing argument  for the cohomology of any degree.
The difficulty that one meets here is that, while we had before  a singular curve in the moduli space of stable maps $f$ to $X$, and only needed
to modify the fibers $f_z$ in a Zariski open neighborhood of the singular points of $C$ so as to make them unobstructed,  one would need
to do a similar construction for a  higher dimensional variety $Z$ with a possibly positive dimensional singular locus.
In this direction, let us note the following generalization  of Theorem \ref{theocylRC}.

 \begin{prop} \label{proisole} Let $X$ be smooth projective  rationally connected of dimension $n$ and let
 $Z$ be a variety of dimension $n-2$ with isolated singularities. Let
 $$f: Y\rightarrow X,\,\,p:Y\rightarrow Z$$
 be a family of stable maps with value in $X$ parameterized by $Z$.
 Then for any $k$
 $${\rm Im}\,(f_*\circ p^*: H_{k-2}(Z,\mathbb{Z})\rightarrow H_k(X,\mathbb{Z}))$$
 is contained in $\widetilde{N}_{1,{\rm cyl}}H^{2n-k}(X,\mathbb{Z})$.

 For $k=n$,
 ${\rm Im}\,(f_*\circ p^*: H_{n-2}(Z,\mathbb{Z})\rightarrow H_n(X,\mathbb{Z}))$ is contained in
 $\widetilde{N}^{1}H^{n}(X,\mathbb{Z})$.
 \end{prop}
\begin{proof} The second statement is implied by the first using Lemma \ref{lequisert}.
Using the fact that the singularities of $Z$ are isolated, we apply the same construction as in the proof of Theorem \ref{theocylRC} of gluing very free curves to  the  fiber $f_z:Y_z\rightarrow X$, getting a modified family
\begin{eqnarray}\label{eqfamilyprime} f': Y'\rightarrow X,\,\,p':Y'\rightarrow Z'
\end{eqnarray}
of stable maps to $X$ parameterized by a variety $Z'\stackrel{\tau}{\rightarrow} Z$ which is birational to $Z$ and isomorphic to $Z$ near ${\rm Sing}\,Z$ with the following properties:

(a)  The cylinder homomorphism ${f'}_*\circ {p'}^*:H_{k-2}(Z',\mathbb{Z})\rightarrow H_k(X,\mathbb{Z})$ coincides with
$f_*\circ p^*\circ \tau_*$.

(b)  The moduli space $M$ of stable maps to $X$ is smooth at any point $f'_z:Y'_z\rightarrow X$, where $z$ is a singular point of $Z'$ (or equivalently $Z$), hence at the  point $f'_z$ for $z$ general in $Z'$.

We now conclude as follows: first of all, we reduce to the case where the maps are  embeddings (for example by replacing $X$ by $X\times\mathbb{P}^r$), so that the stable maps are automorphism free. We then consider the universal  deformation of $f'_z,\,z\in Z'$, given by a family of automorphism free
stable maps \begin{eqnarray}\label{eqfamilyM}f_M: Y_M\rightarrow X,\,\,p_M:Y_M\rightarrow M
\end{eqnarray}
parameterized by $M$. Using the automorphism free assumption, we have a classifying morphism $g: Z'\rightarrow M$, such that
the family (\ref{eqfamilyprime}) is obtained from  the family (\ref{eqfamilyM}) by base change under $g$.
We know that $M$ is smooth near $g({\rm Sing}\,Z')$, so we can introduce a desingularization $\widetilde{M}$ of $M$, and a modification
$\tau':\widetilde{Z}'\rightarrow Z'$ which is an isomorphism over ${\rm Sing}\,Z'$, such that
the rational map $g:Z'\dashrightarrow \widetilde{M}$ induces a morphism
$$\tilde{g}:\widetilde{Z}'\rightarrow \widetilde{M}.$$
Over the desingularized moduli space $\widetilde{M}$, we have
the pulled-back family
\begin{eqnarray}\label{eqfamilyMtilde}\tilde{f}_{M}: Y_{\widetilde{M}}\rightarrow X,\,\,\tilde{p}_M:Y_{\widetilde{M}}\rightarrow {\widetilde{M}},
\end{eqnarray}
and over  $\widetilde{Z}'$, we have the family
\begin{eqnarray}\label{eqfamilyprimetilde} \tilde{f}': \widetilde{Y}'\rightarrow X,\,\,\tilde{p}':\widetilde{Y}'\rightarrow \widetilde{Z}',
\end{eqnarray}
which is deduced either from (\ref{eqfamilyMtilde}) by base-change under $\tilde{g}$ or  from (\ref{eqfamilyprime}) by base-change under
$\tau'$.
We conclude that
$$ \tilde{f}'_*\circ (\tilde{p}')^*=
 \tilde{f}_{M*}\circ \tilde{p}_M^*\circ \tilde{g}_*:H_{k-2}(\widetilde{Z}',\mathbb{Z})\rightarrow H_{k}(X,\mathbb{Z}),$$
and, as $\widetilde{M}$ is smooth, we get that ${\rm Im}\,\tilde{f}'_*\circ (\tilde{p}')^*\subset \widetilde{N}_{1,{\rm cyl}} H_k(X,\mathbb{Z})$.
Finally, we also have by (a)
$$ \tilde{f}'_*\circ (\tilde{p}')^*=f_*\circ p^*\circ (\tau\circ \tau')_*:H_{k-2}(\widetilde{Z}',\mathbb{Z})\rightarrow H_{k}(X,\mathbb{Z}),$$
and, as $\tau\circ \tau': \widetilde{Z}'\rightarrow Z$ is proper birational and an isomorphism over ${\rm Sing}\,Z$, the map
$$(\tau\circ \tau')_*:H_{k-2}(\widetilde{Z}',\mathbb{Z})\rightarrow H_{k-2}({Z},\mathbb{Z})$$
 is surjective.
It follows that ${\rm Im}\,f_*\circ p^*\subset \widetilde{N}_{1,{\rm cyl}} H_k(X,\mathbb{Z})$.
\end{proof}
Our last question concerns the representability of the Abel-Jacobi isomorphism for $1$-cycles on rationally connected threefolds (we refer here to \cite{vialetco} for a general discussion of the motivic nature of $J^3(X)$).
As discussed in Section \ref{secaj}, another way of stating Theorem \ref{theomain} or its generalization \ref{theomainvrai} is to say
that, if $X$ is a rationally connected  manifold of dimension $n$,
there exist  a curve $C$ and a codimension  $n-1$ cycle $\mathcal{Z}\in{\rm CH}^{n-1}(C\times X)$
such that the lifted Abel-Jacobi map
$$\widetilde{\Phi}_\mathcal{Z}: J(C)\rightarrow J(N^{n-2}H^{2n-3}(X,\mathbb{Z})_{\rm tf})$$
is surjective with {\it connected} fibers. (When $n=3$, we already mentioned that $N^{1}H^{3}(X,\mathbb{Z})_{\rm tf}=H^{3}(X,\mathbb{Z})_{\rm tf}$.)

Note that it was proved in \cite{voisininvent} that, even for $X$ rationally connected of dimension $3$,there does not necessarily exist
a universal codimension $n-1$ cycle
$$\mathcal{Z}_{\rm univ}\in{\rm CH}^{n-1}(J(N^{n-2}H^{2n-3}(X,\mathbb{Z})_{\rm tf})\times X)$$
such that the induced lifted Abel-Jacobi map
$$\widetilde{\Phi}_{\mathcal{Z}}:J(N^{n-2}H^{2n-3}(X,\mathbb{Z})_{\rm tf})\rightarrow J(N^{n-2}H^{2n-3}(X,\mathbb{Z})_{\rm tf})$$
is the identity.
However, the following question remains open.
\begin{question} Let $X$ be a rationally connected manifold of dimension $n$. Does there exist a smooth  projective manifold
$M$ and a codimension ${n-1}$-cycle
$$\mathcal{Z}_M\in{\rm CH}^{n-1}(M\times X)$$
such that the map
$$\widetilde{\Phi}_{\mathcal{Z}_M}:{\rm Alb}\,M\rightarrow J(N^{n-2}H^{2n-3}(X,\mathbb{Z})_{\rm tf})$$
is an isomorphism?
\end{question}
In practice, the answer is yes for Fano threefolds, at least for generic ones. For example, one can use the Fano surface of lines
for the cubic threefold (see \cite{clegri}), and similarly for the quartic double solid \cite{welters}. For quartic threefolds, the surface of conics works (see \cite{letzia}).

The motivation for asking this question is the following:
\begin{prop}\label{profinalpourquestionaj} If $X$ admits a cohomological decomposition of the diagonal,
(in particular, if $X$ is stably rational), there exist
a smooth  projective manifold
$M$ and a codimension-${n-1}$ cycle
$$\mathcal{Z}_M\in{\rm CH}^{n-1}(M\times X)$$
such that the Abel-Jacobi map
$$\widetilde{\Phi}_{\mathcal{Z}_M}:{\rm Alb}\,M\rightarrow J(H^{2n-3}(X,\mathbb{Z})_{\rm tf})$$
is an isomorphism
\end{prop}
(Note that $N^{n-2}H^{2n-3}(X,\mathbb{Z})_{\rm tf}=H^{2n-3}(X,\mathbb{Z})_{\rm tf}$ under the same assumption.)
\begin{proof}[Proof of Proposition \ref{profinalpourquestionaj}] It follows from
Theorem \ref{theojems} that there exist
a (nonnecessarily connected) smooth projective variety
$Z$ of dimension $n-2$ and a family of $1$-cycles
$$\Gamma\in {\rm CH}^{n-1}(Z\times X)$$
such that
$$\Gamma_*:{\rm Alb}(Z)\rightarrow J(H^{2n-3}(X,\mathbb{Z})_{\rm tf})$$
is surjective with a right   inverse ${\Gamma'}^*:J(H^{2n-3}(X,\mathbb{Z})_{\rm tf})\rightarrow {\rm Alb}(Z)$. We now have
 the following lemma.
\begin{lemm} Let $Z$ be a smooth projective variety of dimension $n-2$ and let $A\subset {\rm Alb}(Z)$ be an abelian subvariety.
Then there exists a smooth projective variety $Z'$ and a $0$-correspondence
$\gamma'\in{\rm CH}^{n-2}(Z'\times Z)$ inducing an isomorphism
$\gamma'_*:{\rm Alb}\,Z'\cong A\subset {\rm Alb}\,Z$.
\end{lemm}
\begin{proof} Suppose first that $Z$ is connected of dimension $1$. Then  for $N$ large enough,
the Abel map
$$f:Z^{(N)}\rightarrow {\rm Alb}(Z)$$
is a projective bundle. Let now
$Z':=f^{-1}(A)$. One has ${\rm Alb}(Z')\cong A$, and we can take for $\gamma'$ the restriction
to $Z'\times Z$ of the natural incidence correspondence
$I\subset Z^{(N)}\times  Z$.

For the general case, we quickly reduce, using the Lefschetz theorem on hyperplane sections, to the case where $Z$ is a connected surface. Then we consider
a Lefschetz pencil $\widetilde{Z}\rightarrow \mathbb{P}^1$ of ample curves on $Z$.
Let $Z_0$ be a smooth projective model of  ${\rm Pic}^0(\widetilde{Z}/\mathbb{P}^1)$. Then $Z_0$    is birational to ${\rm Pic}^1(\widetilde{Z}/\mathbb{P}^1)$
 using one of the base-points, and thus admits a natural correspondence
 $\gamma\in{\rm CH}^2(Z_0\times\widetilde{Z})$. It is immediate to check that
 $$\gamma_*:{\rm Alb}(Z_0)\rightarrow {\rm Alb}(Z)$$
 is an isomorphism.  Let $a:Z_0\rightarrow {\rm Alb}(Z_0)$ be the Albanese map. We claim that, denoting
 $Z'_u:=a^{-1}(A_u)$, where $A_u$ is a generic translate of $A$ in ${\rm Alb}\,(Z)$, $Z'_u$ is smooth and we have
 $${\rm Alb}(Z'_u)\cong A.$$
As $Z_0$ is smooth and $A_u$ is smooth, the smoothness of   $a^{-1}(A_u)$ for a general translate $A_u$ of $A$ follows from standard transversality arguments.  For the second point, we observe that, by definition, a Zariski open set of  $Z_0$ is fibered over
$\mathbb{P}^1$  into Jacobians $J(\widetilde{Z}_t),\,t\in \mathbb{P}^1$, and that the natural map
 $J(\widetilde{Z}_t)\rightarrow {\rm Alb}(Z)={\rm Alb}(\widetilde{Z})$ has connected fiber isomorphic to $J({\rm Ker}\,(H^1(\widetilde{Z}_t,\mathbb{Z})\rightarrow H^3(\widetilde{Z},\mathbb{Z})_{\rm tf}))$. Here the connectedness of the fibers indeed
 follows from the Lefschetz theorem on hyperplane sections which says that the Gysin morphism
 $$H^1(\widetilde{Z}_t,\mathbb{Z})\rightarrow H^3(\widetilde{Z},\mathbb{Z})_{\rm tf}$$
 is surjective (see the discussion in Section \ref{secmaterial}).
 It follows that a Zariski open set of  $Z'_u$ is fibered over $\mathbb{P}^1\times A_u$ into connected abelian varieties $J({\rm Ker}\,(H^1(\widetilde{Z}_t,\mathbb{Z})\rightarrow H^3(\widetilde{Z},\mathbb{Z})_{\rm tf}))$.
 On the other hand, by the Deligne global invariant cycle theorem,  there is no
 nonconstant morphism from $J({\rm Ker}\,(H^1(\widetilde{Z}_t,\mathbb{Z})\rightarrow H^3(\widetilde{Z},\mathbb{Z})_{\rm tf}))$
 to a fixed abelian variety. It follows that ${\rm Alb}\,Z_u=A$.

Finally, we have $Z'_u\subset Z_0$ and $Z_0$ has a natural correspondence to $Z$, so combining both we get
a natural correspondence $\gamma'$ between $Z'_u$ and $Z$, inducing the morphism
$${\rm Alb}(Z'_u)\cong A\subset {\rm Alb}\,Z.$$
 Then $\Gamma\circ \gamma'$ produces the desired  correspondence.
\end{proof}
We apply this lemma to $A:={\rm Im}\,(\Gamma^*: J(H^{2n-3}(X,\mathbb{Z})_{\rm tf})\rightarrow {\rm Alb}(Z))$. We thus get
a smooth projective variety $Z'$ with Albanese variety isomorphic to $J(H^{2n-3}(X,\mathbb{Z})_{\rm tf})$ and
cycle $\Gamma':=\Gamma\circ \gamma'\in {\rm CH}^{n-1}(Z'\times X)$ which induces the isomorphism
$\Gamma_*\circ \gamma_*:{\rm Alb}\,Z'\cong J(H^{2n-3}(X,\mathbb{Z})_{\rm tf})$.
\end{proof}

CNRS, IMJ-PRG,
 4 Place Jussieu,  Case 247,
75252 Paris Cedex 05,  FRANCE,
claire.voisin@imj-prg.fr
    \end{document}